\documentclass[10pt, reqno]{amsart}

\usepackage{hyperref}
\hypersetup{
	colorlinks=true,
	citecolor=blue,
	linkcolor=blue,
	filecolor=magenta,      
	urlcolor=cyan,
}

\usepackage{setspace}
\usepackage[capitalize,noabbrev,nameinlink]{cleveref}
\usepackage{fullpage}
\usepackage{amsmath, amsthm, amssymb, graphicx, dsfont, stmaryrd, extarrows, enumitem}
\usepackage[T1]{fontenc}
\usepackage{textcomp}
\usepackage{lmodern}
\usepackage{microtype} 
\usepackage{amscd}
\usepackage{mathrsfs}
\usepackage{tikz-cd}
\usetikzlibrary{matrix,patterns}
\usepackage[colorinlistoftodos]{todonotes}
\usepackage{color}
\usepackage{bbm}
\usepackage{verbatim}
\usepackage[mathscr]{euscript}
\usepackage{bm}

\numberwithin{equation}{section}
\theoremstyle{plain}
\newtheorem{thm}[equation]{Theorem}
\newtheorem*{thm*}{Theorem}

\newtheorem{thmx}{Theorem}

\newtheorem{prop}[equation]{Proposition}
    
\newtheorem{cor}[equation]{Corollary}       
\newtheorem{lem}[equation]{Lemma}

\theoremstyle{definition} 
\newtheorem{defn}[equation]{Definition}

\newtheorem{rem}[equation]{Remark}

\newtheorem{chunk}[equation]{}
\crefformat{chunk}{(#2#1#3)}

\newcommand{\Z}{\mathbb{Z}}

\newcommand{\Hom}{\mathrm{Hom}}

\newcommand{\msf}[1]{\mathsf{#1}}

\newcommand{\mc}[1]{\mathcal{#1}}
\newcommand{\mrm}[1]{\mathrm{#1}}
\newcommand{\mbb}[1]{\mathbb{#1}}
\newcommand{\scr}[1]{\mathscr{#1}}

\newcommand{\T}{\mathsf{T}}
\newcommand{\U}{\mathsf{U}}

\newcommand{\1}{\mathds{1}}

\renewcommand{\mod}[1]{\mathrm{Mod}_{#1}}

\newcommand{\A}{\mathsf{A}}
\newcommand{\B}{\mathsf{B}}
\newcommand{\C}{\mathsf{C}}
\newcommand{\D}{\mathsf{D}}

\renewcommand{\mod}[1]{\msf{mod}(#1)}
\newcommand{\Mod}[1]{\msf{Mod}(#1)}
\newcommand{\Flat}[1]{\msf{Flat}(#1)}
\newcommand{\ab}{\msf{Ab}}
\newcommand{\ti}{\textit}
\renewcommand{\tilde}[1]{\widetilde{#1}}
\newcommand{\rlim}{\varinjlim}

\renewcommand{\t}{\text}
\renewcommand{\c}{\mrm{c}}
\newcommand{\down}{{\downarrow}}
\newcommand{\up}{{\uparrow}}
\newcommand{\op}{\mrm{op}}

\renewcommand{\Flat}[1]{\msf{Flat}(#1)}
\renewcommand{\flat}[1]{\msf{flat}(#1)}

\newcommand{\y}{\msf{y}}
\renewcommand{\bar}{\overline}

\renewcommand{\hat}{\widehat}

\renewcommand{\phi}{\varphi}

\title{Definable functors between triangulated categories}

\makeatletter
\@namedef{subjclassname@2020}{\textup{2020} Mathematics Subject Classification}
\makeatother

\begin{document}
\author{Isaac Bird}
\address[Bird]{Department of Algebra, Faculty of Mathematics and Physics, Charles University in Prague, Sokolovsk\'{a} 83, 186 75 Praha, Czech Republic}
\email{bird@karlin.mff.cuni.cz}
\author{Jordan Williamson}
\address[Williamson]{Department of Algebra, Faculty of Mathematics and Physics, Charles University in Prague, Sokolovsk\'{a} 83, 186 75 Praha, Czech Republic}
\email{williamson@karlin.mff.cuni.cz}

\begin{abstract}
	We systematically develop the theory of definable functors between compactly generated triangulated categories. Such functors preserve pure triangles, pure injective objects, and definable subcategories, and as such appear in a wide range of  algebraic and topological settings. 
    Firstly we investigate and characterise purity preserving functors from a triangulated category into a finitely accessible category with products, which we term coherent functors. This yields a new property for the restricted Yoneda embedding as the universal coherent functor. We build upon the utility of coherent functors to provide several equivalent conditions for an additive, not necessarily triangulated, functor between triangulated categories to be definable: a functor is definable if and only if it preserves filtered homology colimits and products, if and only if it uniquely extends along the restricted Yoneda embedding to a definable functor between the corresponding module categories. We apply these results to the functoriality of the Ziegler spectrum, an object of study in pure homological algebra and representation theory. 
\end{abstract}
\maketitle
\setcounter{tocdepth}{1}
\tableofcontents


\section{Introduction}
Classification problems are one of the central pillars of mathematics, but are frequently highly intractable. For example, it is hopeless to try to classify topological spaces up to homotopy equivalence. Hoping to understand all modules over a ring is equally inaccessible.
To overcome such difficulties, one often adopts the two-fold strategy of weakening the notion of equivalence by passing to a stable setting, and then transporting information from a more well understood and tractable situation in order to gain insight into the original problem.

For instance, rather than classifying topological spaces up to homotopy equivalence, one weakens to the notion of stable equivalence and then seeks a classification in this new setting. The stable homotopy category of spectra is the resulting category, and it is an example of a compactly generated triangulated category. In algebra, taking the stable category of modules offers an analogous construction, as does passing to the derived category of chain complexes. Both of these are further examples of compactly generated triangulated categories.

Much of the information of the stable homotopy category of spectra is obtained through understanding various localisations, such as at the Morava $E$-theories, or via descent to more algebraically favourable situations. Knowledge of these smaller and simpler pieces may then be transported back to the whole category. As such, having a method for the transfer of information between different settings provides insight into the structure of compactly generated triangulated categories.

In this paper, we introduce definable functors in the setting of compactly generated triangulated categories, thus providing a means for transporting information concerning the structure of purity. The pure structure enables one to reduce many classification problems from studying proper classes to  understanding the structure of a particular set of objects. There are many examples of classification problems of active research interest which are related to the pure structure: 
\begin{itemize}
	\item smashing subcategories in stable homotopy theory~\cite{krsmash, krcq},
	\item silting and cosilting objects in representation theory~\cite{AMV,laking,lv,NSZ,saorstov}, 
	\item the homological, and thus Balmer, spectrum in tensor-triangular geometry~\cite{Balmerhomsupp,Balmernilpotence,bks,birdwilliamsonhomological}, 
	\item rank functions on triangulated categories with connections to stability conditions~\cite{chuanglazarev,conde2022functorial}, 
	\item the Ziegler spectrum in model theoretic algebra~\cite{bg,birdgfc,dualitypairs,prestgarkusha},
\end{itemize}
to name a few, are all encoded within the pure structure of compactly generated triangulated categories. Thus, understanding the pure structure, and significantly for this paper, its transport, is widely applicable. 

The striking thing about the structure of purity in compactly generated triangulated categories is that it is a property of the additive, rather than triangulated, structure of the category. To understand this, and to motivate the pure structure, let us consider the following question:
given a compactly generated triangulated category $\T$, how much of $\T$ can be recovered from its compact objects, $\T^{\c}$?

If one uses the triangulated structure, then, by definition, one can recover all of $\T$ from $\T^{\c}$; any object of $\T$ can be built from compact objects using triangles, sums, and retracts. But if one views $\T^{\c}$ as a small additive category and neglects the triangulated structure, one obtains a different outcome.

Here, one can take the ind-completion of $\T^{\c}$. Recall that there is a fully faithful embedding
\[
\y\colon\T^{\c}\to \Mod{\T^{\c}}=\msf{Add}((\T^{\c})^{\op},\ab)
\]
of $\T^{\c}$ into the category of right $\T^{\c}$-modules. This associates $\T^{\c}$ with the finitely presented projective modules in $\Mod{\T^{\c}}$. In analogy to the Govorov-Lazard theorem for modules over a ring, when we close this category of finitely presented projective modules under filtered colimits, we obtain the category of flat functors in $\Mod{\T^{\c}}$, denoted $\Flat{\T^{\c}}$, which may be identified with the category of cohomological functors $(\T^{\c})^{\op}\to\ab$. This is the ind-completion of $\T^{\c}$, and its finitely presented objects are equivalent to $\T^{\c}$.

Now, $\Flat{\T^{\c}}$ is, in many ways, extremely different to $\T$, even though they have the same small objects. Firstly, categorically they are very different, as $\T$ will not have true filtered colimits, while $\Flat{\T^{\c}}$ does; moreover $\Flat{\T^{\c}}$ is an exact, but not triangulated category. We may further see the difference between $\T$ and $\Flat{\T^{\c}}$ through the restricted Yoneda embedding $\y\colon X\mapsto\Hom_{\T}(-,X)\vert_{\T^{\c}}$, whose image lies within $\Flat{\T^{\c}}$. We discuss this functor more formally below. The maps which vanish under $\y$ form precisely the ideal of phantom maps in $\T$, hence $\y$ is faithful if and only if $\T$ is phantomless. Furthermore, in contrast to the classic case of the homotopy category of spectra studied by Adams~\cite{Adams}, $\y$ is seldom essentially surjective, nor full. Fullness fails for the derived category of rings as simple as $\mbb{C}[x,y]$, as illustrated in \cite[\S 6]{NeemanBrownAdams}, while examples of rings for which $\y$ is not essentially surjective can be found in \cite{CKN}. The question of when $\y$ is essentially surjective and full, as it is in the case of the homotopy category of spectra, is discussed in \cite[Theorem 11.8]{bel}, and a generalisation of Adams's result to the case when the compact objects form a countable category (as is true for spectra) is in \cite[\S5]{NeemanBrownAdams}. However, we put no restrictions whatsoever on our compactly generated triangulated category, so we do not assume either fullness or essential surjectivity of $\y$. 

Despite the categorical differences and distance that $\T$ and $\Flat{\T^{\c}}$ may be from being equivalent, the two approaches of triangulated and additive completion converge in one area: the pure structure. The pure structure of $\T$ is equivalent to the pure structure of $\Flat{\T^{\c}}$. This equivalence is at the heart of the study of purity in compactly generated triangulated categories, and, like many before us, we use it throughout. Yet to our knowledge, this paper is the first which uses the equivalence to explore functoriality of the pure structures without appealing to triangulated functors. Definable functors will be precisely those functors which transfer this pure structure between compactly generated triangulated categories.

So far, we have mentioned, but not defined, the pure structures on either $\T$ or $\Flat{\T^{\c}}$. Let us now be more precise. Following~\cite{krsmash}, a morphism $f$ in $\T$ is a \ti{pure monomorphism} if $\y(f)$ is a monomorphism in $\Mod{\T^{\c}}$. A triangle $X\to Y\to Z\to \Sigma X$ in $\T$ is said to be \ti{pure} if the map $X\to Y$ is a pure monomorphism. This is the case if and only if $Z\to \Sigma X$ is a phantom map. An object $X$ is said to be \ti{pure injective} if and only if every pure monomorphism $X\to Y$ is split. The pure triangles, with the pure injective objects, are two of the three ingredients which determine the pure structure of $\T$. The third are definable subcategories, which are the classes of objects which are annihilated by functors of the form $\msf{coker}(\Hom_{\T}(f,-))$, where $f$ is a morphism of compact objects. Using analogous definitions, one can define the same concepts in $\Flat{\T^{\c}}$, see~\cite{dac}.
\subsection*{The Yoneda embedding as the universal coherent functor}
To relate purity in $\T$ to purity in $\Flat{\T^{\c}}$, one uses the restricted Yoneda embedding $\y\colon\T\to\Flat{\T^{\c}}$, sending $X$ to $\Hom_{\T}(-,X)\vert_{\T^{\c}}$. The key point is that $\y$ provides an equivalence between these two theories of purity: a triangle is pure in $\T$ if and only if its image under $\y$ is a pure short exact sequence in $\Flat{\T^{\c}}$. Moreover, an object $X\in\T$ is pure injective if and only if $\y X$ is pure injective in $\Flat{\T^{\c}}$. One can also establish a bijection between definable subcategories of $\T$ and definable subcategories of $\Flat{\T^{\c}}$. In other words, the pure structure of $\T$ completely determines, and is determined by, the pure structure of $\Flat{\T^{\c}}$ through $
\y$. 

It is well known that $\y$ is the universal homological and coproduct preserving functor into an AB5 abelian category~\cite{krsmash}, but the purity preservation has nothing to do with $\y$'s homological nature. So, one may wonder whether $\y$ also has a universal property in relation to purity preservation. The existence of such a universal property is the first main theorem of this paper, which we now turn to formulating.

Let $\scr{A}$ be a finitely accessible category with products; this is precisely the setting one needs to be able to discuss the pure structure in the non-triangulated setting. We say that a functor $H\colon\T\to\scr{A}$
is \ti{coherent} if $H$ sends pure triangles to pure exact sequences, and preserves coproducts and products. We immediately see that $\y$ is a coherent functor, and that coherent functors are ubiquitous; for example homology in a derived category or homotopy in the stable homotopy category are both coherent. Coherent functors with target $\ab$ were introduced by Krause~\cite{krcoh}, and as such are a particular case of our more general notion, as discussed in \cref{abcoh}. We note that we do not require coherent functors to be homological.

Let us recall the notion of a definable functor in the world of finitely accessible categories. By definition, this is a functor which commutes with filtered colimits and products. These two properties ensure that definable functors preserve pure exact sequences and pure injective objects, and provide a means of pushing forward, and pulling back, definable subcategories. Definable functors between finitely accessible categories, unlike their triangulated counterparts that we introduce in this paper, are well understood, and have undergone significant study since their inception in the 1990s. Their utility can be seen in~\cite{birdgfc,gregprest,kredc,prestinterp,dac} for instance.

With that said, let us state our first main theorem.

\begin{thmx}(\ref{uniprop})\label{thmA}
	Let $H\colon\T\to\scr{A}$ be a coherent functor. Then there is a unique definable functor $\hat{H}\colon\Flat{\T^{\c}}\to\scr{A}$ between finitely accessible categories such that
	\[
	\begin{tikzcd}[row sep=1cm, column sep=1.5cm]
		\T \arrow[d,swap, "\y"] \arrow[dr, "H"] &  \\
		\Flat{\T^{\c}} \arrow[r,swap, "\hat{H}"] & \scr{A}
	\end{tikzcd}
	\]
	commutes. Consequently, $\y\colon\T\to\Flat{\T^{\c}}$ is the universal coherent functor.
\end{thmx}

The proof of this theorem occupies the bulk of \cref{sec:universalprop}, so let us give a brief overview of how it proceeds. The first stage is embedding the category of coherent functors $\T\to\scr{A}$ into a certain category of exact functors. We then use the yoga of finitely accessible categories introduced by Crawley-Boevey in~\cite{CrawleyBoevey} to prove this category of exact functors is equivalent to the category of definable functors $\Flat{\T^{\c}}\to\scr{A}$. This gives us a candidate for $\hat{H}$. However, it must be shown that the above diagram commutes. By purely embedding $\Flat{\T^{\c}}$ into its universal locally coherent `envelope' via two further Yoneda embeddings, we unravel several equivalences of categories to show that this candidate functor is precisely what we need and is unique.

We then establish some further properties of $\hat{H}$. In the case when $\scr{A}$ has kernels, we may further extend $\hat{H}$ to a unique definable functor $\bar{H}\colon\Mod{\T^{\c}}\to\scr{A}$, see \cref{moduleextension}. This also shows that $\y$ is the universal coherent functor into finitely accessible Grothendieck categories.
We note that when $H$ is homological, this extension $\bar{H}$ agrees with the well established extensions of Beligiannis~\cite{bel} and Krause~\cite{krsmash}. 

\subsection*{Definable functors and their lifts}

As of yet, we have made no further mention of purity preserving functors between triangulated categories, but it transpires that coherent functors are precisely what one needs to make an appropriate definition. We  categorify the fact that purity in $\T$ and purity in $\Flat{\T^{\c}}$ are equivalent to both motivate, and define, a definable functor of triangulated categories. This is witnessed by the first condition in the following theorem.

\begin{thmx}(\ref{enhancement})\label{intro:deftheorem}
	Let $F\colon\T\to\U$ be an additive functor between compactly generated triangulated categories. Then the following are equivalent:
	\begin{enumerate} 
		\item there exists a unique definable functor $\hat{F}\colon\Flat{\T^{\c}}\to\Flat{\U^{\c}}$ of finitely accessible categories such that 
		\[
		\begin{tikzcd}
			\T \arrow[r, "F"] \arrow[d,swap, "\y"] & \U \arrow[d, "\y"] \\
			\Flat{\T^{\c}} \arrow[r, swap, "\hat{F}"] & \Flat{\U^{\c}}
		\end{tikzcd}
		\]
		commutes;
		\item $F$ preserves products and filtered homology colimits;
		\item $F$ preserves coproducts, products, and pure triangles.
	\end{enumerate}
	We say that $F$ is \emph{definable} if it satisfies any of these equivalent conditions.
\end{thmx}

We note that condition (1) asks for the \emph{existence} of a functor $\hat{F}$ extending $F$ which is also definable; in general, given an arbitrary additive functor $G\colon \T \to \U$ there need not be a lift $\hat{G}\colon \Flat{\T^\c} \to \Flat{\U^\c}$ extending $G$. 

Thus the above provides the appropriate notion of a definable functor in the triangulated world: not only does it replicate the properties of definable functors between finitely accessible categories, it also shows that any definable functor between triangulated categories induces a definable functor between finitely accessible categories. This extends and categorifies the relationship between purity in triangulated and finitely accessible categories from objects (the triangulated categories themselves) to morphisms (the definable functors).

At this stage we should highlight that the concept of a definable \emph{triangulated} functor already exists in the literature, due to Beligiannis~\cite{Beligiannis}: a triangulated functor is Beligiannis definable if and only if it preserves products and coproducts. We show in \cref{triangulatedversion} that a functor is Beligiannis definable if and only if it is a triangulated functor which is definable in the sense of \cref{intro:deftheorem}.

We demonstrate that definable functors as introduced in \cref{intro:deftheorem} provide the correct notion of purity preserving functors in this setting. Not only do they preserve pure triangles, but they also preserve pure injective objects, including endofinite objects, and enable the pullback and pushforward of definable subcategories. Indeed, since $\y$ is coherent, these properties follow from the fact that $F$ induces a unique definable functor $\hat{F}$. 

However, it is often more convenient to work with finitely accessible Grothendieck categories than arbitrary finitely accessible categories; in particular, the obvious such category to consider is $\Mod{\T^{\c}}$. Indeed, in much of the literature, $\Flat{\T^{\c}}$ is viewed secondarily to $\Mod{\T^{\c}}$ when discussing purity in $\T$. However, in terms of definable functors, it does not make any difference whatsoever.

\begin{thmx}(\ref{leftexactextension}, \ref{barFhasleftadjoint})\label{intro:extension}
	Let $F\colon\T\to\U$ be a definable functor. Then there is a unique left exact definable functor $\bar{F}\colon\Mod{\T^{\c}}\to\Mod{\U^{\c}}$ such that $\bar{F}\circ \y = \y\circ F$. Furthermore, $\bar{F}$ admits an exact left adjoint.
\end{thmx}

Since definable functors need not be triangulated, we may not apply Brown representability to obtain the existence of adjoints on the level of triangulated categories. The above result is a partial replacement for this: one will always obtain the existence of a left adjoint on the level of module categories, irrespective of whether the original definable functor admits any adjoints. 

At this stage we make clear that definable functors are ubiquitous, both triangulated and otherwise. In \cref{sec:examples}, we give many explicit examples of such functors arising in homotopy theory, representation theory, homological algebra, and tensor-triangular geometry.

A discussion about the preservation of the pure structure would be incomplete without considering induced maps on the Ziegler spectrum. This is a topological space $\msf{Zg}(\T)$ which encodes the pure structure of $\T$: its points are isomorphism classes of indecomposable pure injective objects, while its closed sets biject with definable subcategories. As such it can be seen as topologically parametrising all finite type quotients of $\Mod{\T^{\c}}$. The Ziegler spectrum of compactly generated triangulated categories has long been an object of study, such as the case of stable module categories of group rings, and other quasi-Frobenius rings, see for instance~\cite{bg, krbook, psl}. It has also seen use in the derived categories of finite dimensional algebras~\cite{discrete}, and the tools of definable subcategories and indecomposable pure injectives are a mainstay in representation theoretic questions for big objects. It has also been shown that this space finds its use in tensor-triangular geometry \cite{birdwilliamsonhomological}. The following provides a sample application of the theory developed in this paper, by demonstrating how definable functors induce maps between Ziegler spectra.
\begin{thmx}(\ref{thm:definducedonzeigler})
	Let $F\colon\T\to\U$ be a definable functor, and suppose that $F$ is full on pure injective objects. Then $F$ gives a closed, continuous map $F\colon \msf{Zg}(\T) \to \msf{Zg}(\msf{U})$, which moreover restricts to a homeomorphism
	\[
	\msf{Zg}(\T)\setminus\msf{K} \xrightarrow{\sim} \msf{Zg}(\U)\cap\msf{Def}(\msf{Im}(F))
	\]
    where $\msf{K} = \{X \in \msf{pinj}(\T) : FX = 0\}$.
\end{thmx}

The framework developed in this paper has already found use in applications of purity to the study of triangulated categories, for instance in the connection with tensor-triangular geometry~\cite{birdwilliamsonhomological}, and interactions between purity and rank functions~\cite{BWZ}. One can also use the theory of definable functors to generalise Balmer’s functoriality of the homological spectrum in tensor-triangular geometry~\cite{Balmerhomsupp}. We intend to return to this in future work.

\subsection*{Acknowledgements}
The first author is grateful for the hospitality of the BIREP group at Universit\"{a}t Bielefeld, where the initial stages of this work were carried out. Both authors are grateful to S. Balchin, G. Garkusha, M. Prest, and A. Zvonareva for their helpful comments at various stages of this project.

Both authors were supported by the grant GA~\v{C}R 20-02760Y from the Czech Science Foundation, the project PRIMUS/23/SCI/006 from Charles University, and by the Charles University Research Center program UNCE/SCI/022. 

\section{Preliminaries}
In this section we recall the necessary background on purity which we will use repeatedly throughout the paper. We refer the reader to \cite{krsmash, krcoh, krbook, psl, dac} for further details, and alternative exposition. Throughout, all categories are at least preadditive, and thus all functors considered are also assumed to be additive, hence the term `functor' will actually, unless stated otherwise, mean additive functor. If $\A$ and $\B$ are preadditive categories, we write $(\A,\B)$ for the category of additive functors $\A\to\B$.
\subsection{Purity in triangulated categories}
Let $\T$ be a compactly generated triangulated category with compact objects $\T^{\c}$. 
\begin{chunk}\label{universalpropertyofyoneda}
We let $\Mod{\T^{\c}}$ denote the category of additive functors $(\T^{\c})^{\op}\to\ab$, which we call (right) modules. There is a canonical way to view the objects of $\T$ in $\Mod{\T^{\c}}$, via the \emph{restricted Yoneda embedding} 
\[
\y\colon\T\to\Mod{\T^{\c}}, \, X\mapsto \Hom_{\T}(-,X)\vert_{\T^{\c}}.
\]
In fact, $\y$ is the universal coproduct-preserving cohomological functor into an AB5 category, in the sense that if $F\colon\T\to\scr{A}$ is a coproduct preserving cohomological functor with $\scr{A}$ being AB5, then there is a unique $G\colon\Mod{\T^{\c}}\to\scr{A}$ such that $F=G\circ \y$, see \cite[Corollary 2.4]{krsmash}.

Now, $\Mod{\T^{\c}}$ is a locally coherent abelian category, meaning its finitely presented objects form an abelian category. These finitely presented objects, denoted $\mod{\T^{\c}}$, are the modules $f\in\Mod{\T^{\c}}$ which have a presentation
\[
\y A\to\y B\to f\to 0
\]
with $A,B\in\T^{\c}$. Since $\y$ is fully faithful on $\T^{\c}$, it follows that a module $f$ is finitely presented if and only if $f\simeq\msf{coker}(\y \alpha)$ for some morphism $\alpha\in\T^{\c}$.
\end{chunk}

\begin{chunk}\label{definitionpureexact}
The restricted Yoneda embedding $\y$ sends any triangle in $\T$ to a long exact sequence in $\Mod{\T^\c}$. The triangles which are sent to short exact sequences are the \emph{pure} triangles. More precisely, the following are equivalent for a triangle $X\to Y\to Z\xrightarrow{f}\Sigma X$:
\begin{enumerate}
\item the triangle is pure;
\item $0\to \y X\to \y Y\to \y Z\to 0$ is exact in $\Mod{\T^{\c}}$;
\item the map $f$ is \emph{phantom}, that is $\y(f)=0$.
\end{enumerate}
In such a pure triangle, the map $X\to Y$ is called a \emph{pure monomorphism}, and $X$ is said to be a \emph{pure subobject} of $Y$. Similarly, we define a \emph{pure epimorphism} and \emph{pure quotient}. Clearly a map $\alpha$ is a pure monomorphism (resp., epimorphism) if and only if $\y(\alpha)$ is a monomorphism (resp., epimorphism).
\end{chunk}

\begin{chunk}\label{pureinjectiveequivalences}
An object $X\in\T$ is pure injective if any of the following equivalent conditions hold:
\begin{enumerate}
\item any pure monomorphism $X\to Y$ splits;
\item $\y X$ is an injective object in $\Mod{\T^{\c}}$;
\item\label{jl} for any set $I$, the canonical summation map $\bigoplus_I X\to X$ factors through $\bigoplus_IX \to \prod_I X$;
\item the induced map $\Hom_{\T}(A,X)\to\Hom(\y A,\y X)$ is a bijection for all $A \in \T$.
\end{enumerate}
We refer to condition (3) as the \emph{Jensen-Lenzing criterion}. We let $\msf{Pinj}(\T)$ denote the full subcategory of $\T$ consisting of the pure injective objects. There is an equivalence of categories $\y\colon \msf{Pinj}(\T)\xrightarrow{\sim}\msf{Inj}(\T^{\c})$, where the latter is the category of injective objects in $\Mod{\T^{\c}}$. As such, there is only a set of isomorphism classes of indecomposable pure injective objects in $\T$, which we denote by $\msf{pinj}(\T)$.
\end{chunk}

\begin{chunk}\label{coherent}
We also consider the locally coherent abelian category of left modules, which is denoted by $(\T^{\c},\ab)$. The functor \[\msf{h}\colon\T^\op\to (\T^{\c},\ab), \, X\mapsto \Hom_{\T}(X,-)\vert_{\T^{\c}}\] similarly enables us to consider objects of $\T$ as modules. A left module $f$ is finitely presented if and only if it has a presentation $\msf{h}A\to\msf{h}B\to f\to 0$, and we let $\msf{fp}(\T^{\c},\ab)$ denote the abelian subcategory of finitely presented left modules.

A functor $F\colon \T \to \ab$ is \emph{coherent} if it has a presentation 
\[
\Hom_{\T}(A,-)\to\Hom_{\T}(B,-)\to F\to 0
\]
with $A,B\in\T^{\c}$. We let $\msf{Coh}(\T,\ab)$ denote the category of coherent functors. This is an abelian category, and it is an immediate observation that there is an equivalence of categories $\msf{fp}(\T^{\c},\ab)\simeq \msf{Coh}(\T,\ab)$ given by extending and restricting presentations.
\end{chunk}

\begin{chunk}
The categories $\mod{\T^{\c}}$ and $\msf{Coh}(\T,\ab)$ are antiequivalent, as proved in \cite[Lemma 7.2]{krcoh}. The mutually inverse functors are
\[
f\mapsto f^\vee\colon (X\mapsto \Hom(f,\msf{y}X))
\]
and
\[
F\mapsto F^\vee\colon (Y\mapsto \Hom(F,\msf{h}Y))
\]
for any object $X\in\T$ and $Y\in\T^{\c}$. The fact that the latter is the inverse is shown at \cite[Lemma 2.5]{birdwilliamsonhomological}.
\end{chunk}

\begin{chunk}
A full subcategory $\mc{D}\subseteq\T$ is said to be \ti{definable} if there is a set of coherent functors $\Phi\subseteq\msf{Coh}(\T,\ab)$ such that
\[
\mc{D}=\{X\in\T:FX=0 \t{ for all }F\in\Phi\}.
\]
Given a class of objects $\msf{X}\subset\T$, we shall let $\msf{Def}(\msf{X})$ denote the smallest definable subcategory containing $\msf{X}$. 
Definable subcategories are closed under products, pure subobjects, pure quotients, and filtered homology colimits (see \cref{homologycolimits} for more details on the latter); this can be seen from the definition by applying~\cite[Theorem A]{krcoh}. Definable subcategories of $\T$ are uniquely determined by the indecomposable pure injective objects they contain: if $\mc{D}_{1}$ and $\mc{D}_{2}$ are definable, then $\mc{D}_{1}=\mc{D}_{2}$ if and only if $\mc{D}_{1}\cap\msf{pinj}(\T)=\mc{D}_{2}\cap\msf{pinj}(\T)$. In particular, the definable subcategory generated by any nonzero object must contain a nonzero indecomposable pure injective, and for any definable subcategory $\mc{D}$, we have $\mc{D} = \msf{Def}(\mc{D} \cap \msf{pinj}(\T))$.
\end{chunk}

\begin{chunk}\label{definitionZieglerspectrum}
The \ti{Ziegler spectrum} of $\T$, denoted $\msf{Zg}(\T)$, is the topological space whose underlying set is $\msf{pinj}(\T)$, and whose closed sets are $\mc{D}\cap\msf{pinj}(\T)$, where $\mc{D}\subseteq\T$ is a definable subcategory.
A subset of $\msf{pinj}(\T)$ is said to be \ti{Ziegler closed} if it is of the form $\mc{D}\cap\msf{pinj}(\T)$ for some definable subcategory $\mc{D}\subseteq\T$. If $\mc{D}\subseteq\T$ is a definable subcategory, we write $\msf{Zg}(\mc{D})$ for the space $\msf{Zg}(\T)\cap\mc{D}$ equipped with the subspace topology.
\end{chunk}

\begin{chunk}\label{fundamentalcorrespondence}
We shall frequently use what is known as the \emph{fundamental correspondence}, which can be found in its totality in \cite{krcoh}. Given a definable subcategory $\mc{D}\subseteq\T$, set \[\scr{S}(\mc{D})=\{f\in\mod{\T^{\c}}:\Hom(f,\y X)=0 \t{ for all }X\in\mc{D}\}.\] Conversely, given a Serre subcategory $\mc{S}\subseteq\mod{\T^{\c}}$, we let \[\scr{D}(\mc{S})=\{X\in\T:\Hom(f,\y X)=0 \t{ for all }f\in\mc{S}\}.\] The fundamental correspondence states that the following diagram is a commutative square of bijections, where each pair of arrows is mutually inverse. In particular, the rightmost pair can be deduced from the rest of the arrows.
\[
\begin{tikzcd}[column sep=1in]
\{\t{Definable subcategories of } \T \} \arrow[r, shift left = 3pt, "\mc{D}\mapsto \mc{D}\cap\msf{pinj}(\T)"] \arrow[d, shift left = 3pt, "\mc{D}\mapsto\scr{S}(\mc{D})"]& \{\t{Closed subsets of }\msf{Zg}(\T)\} \arrow[l, shift left = 3pt, "\msf{X}\mapsto \msf{Def}(\msf{X})"] \arrow[d, shift left = 3pt]\\
\{\t{Serre subcategories of }\mod{\T^{\c}}\} \arrow[u, shift left = 3pt, "\mc{S}\mapsto \scr{D}(\mc{S})"] \arrow[r, shift left = 3pt, "\mc{S}\mapsto \mc{S}^{\vee}"]& \{\t{Serre subcategories of }\msf{Coh}(\T,\ab)\} \arrow[l,shift left = 3pt, "\mc{S}_{\c}\mapsto \mc{S}_{\c}^{\vee}"] \arrow[u, shift left = 3pt]
\end{tikzcd}
\]
It is important to note that the vertical arrows are order reversing bijections, while the horizontal arrows are order preserving bijections.
\end{chunk}

\begin{chunk}\label{flatequivalences}
The subcategory of cohomological functors in $\Mod{\T^{\c}}$ will be central to our study. Recall from~\cite[\S 2.2]{krsmash} that the following are equivalent for a module $F\in\Mod{\T^{\c}}$: \begin{enumerate}[label=(\arabic*)]
\item $F$ is a flat functor;
\item $F$ is a filtered colimit of finitely presented projective functors;
\item $F$ is a cohomological functor; that is it sends triangles in $\T^\c$ to exact sequences in $\ab$;
\item\label{flatandfpinj} $F$ is fp-injective; that is $\mrm{Ext}^{1}(g,F)=0$ for all $g\in\mod{\T^{\c}}$.
\end{enumerate}
We write $\Flat{\T^{\c}}$ for the full subcategory of modules that satisfy the above conditions. We write $\msf{flat}(\T^\c) := \msf{fp}(\Flat{\T^\c})$ for the finitely presented flat modules, which, by \cite[1.3]{CrawleyBoevey}, coincides with $\Flat{\T^\c}\cap\mod{\T^\c}$. Moreover, there is an equivalence of categories $\y\colon\T^{\c}\xrightarrow{\sim}\msf{flat}(\T^{\c})$.
\end{chunk}

\begin{chunk}\label{appendix.definable}
All of the above definitions concerning purity also hold in finitely accessible categories with products. Recall that a category $\scr{A}$ is finitely accessible if it has filtered colimits, the subcategory of finitely presented objects $\msf{fp}(\scr{A})$ is skeletally small, and every object of $\scr{A}$ is a filtered colimit of finitely presented objects. For example, the category $\Flat{\T^\c}$ is a finitely accessible category with products by the discussion in \cref{flatequivalences}. 

For a finitely accessible category with products $\scr{A}$, a short exact sequence $0\to L\to M\to N\to 0$ is pure if and only if $0\to \y L\to \y M\to \y N\to 0$ is exact in $(\msf{fp}(\scr{A})^{\op},\ab)$, in which case $L$ is a \emph{pure subobject}, and $N$ is a \emph{pure quotient} of $M$.

Definable subcategories of $\scr{A}$ are precisely the full subcategories $\mc{D}\subseteq\scr{A}$ such that $\mc{D}$ is closed under direct limits, products and pure subobjects. There is also a set of indecomposable pure injective objects in $\scr{A}$, which is denoted $\msf{pinj}(\scr{A})$, and definable subcategories are uniquely determined by the indecomposable pure injective objects they contain. Moreover, for a definable subcategory $\mc{D}$, one has $A\in\mc{D}$ if and only if there is a pure monomorphism $0\to A\to \prod_{i}E_{i}$, where $E_{i}\in\mc{D}\cap\msf{pinj}(\scr{A})$, as in \cref{fundamentalcorrespondence}.

The Ziegler spectrum of $\scr{A}$, denoted $\msf{Zg}(\scr{A})$ is defined analogously to the triangulated case: the points are the indecomposable pure injectives, and the closed sets are of the form $\mc{D}\cap\msf{pinj}(\scr{A})$, with $\mc{D}\subseteq\scr{A}$ definable. 
\end{chunk}

\begin{chunk}
The point of this is that the category $\Flat{\T^{\c}}$ is a definable subcategory of $\Mod{\T^{\c}}$. Furthermore, since $\y X\in\Flat{\T^{\c}}$ for all $X\in\T$, we see that $\Flat{\T^{\c}}$ is the smallest definable subcategory of $\Mod{\T^{\c}}$ containing the image of $\y$. In fact, we have more, as the following shows.
\end{chunk}
\begin{lem}\label{purityflats}
Let $\T$ be a compactly generated triangulated category.   
\begin{enumerate}[label=(\arabic*)]
\item\label{purityflats0} A functor $X \in \Mod{\T^\c}$ is flat and pure injective if and only if it is injective.
\item\label{purityflats1} The restricted Yoneda embedding induces an equivalence of categories $\msf{y}\colon\msf{Pinj}(\T)\xrightarrow{\sim}\msf{Pinj}(\Flat{\T^{\c}})$.

\item\label{purityflats2} For any definable subcategory $\mc{D}$ of $\T$, we have $\msf{y}(\mc{D} \cap \msf{pinj}(\T)) = \msf{Def}(\msf{y}\mc{D}) \cap \msf{pinj}(\Flat{\T^{\c}})$.

\item\label{purityflats3} There is a bijection between definable subcategories of $\T$ and those of $\msf{Flat}(\T^{\c})$ given by $\mc{D} \mapsto \msf{Def}(\msf{y}\mc{D})$ and $\widetilde{\mc{D}} \mapsto \msf{y}^{-1}(\widetilde{\mc{D}})$, for $\mc{D}$ a definable subcategory of $\T$, and $\widetilde{\mc{D}}$ a definable subcategory of $\Flat{\T^\c}$.
\item\label{purityflats4} For any definable subcategory $\mc{D}$ of $\T$, there is a homeomorphism $\y\colon \msf{Zg}(\mc{D}) \xrightarrow{\sim} \msf{Zg}(\msf{Def}(\y\mc{D}))$. In particular, there is a homeomorphism $\msf{y}\colon \msf{Zg}(\T) \xrightarrow{\sim} \msf{Zg}(\Flat{\T^\c})$.
\end{enumerate}
\end{lem}
\begin{proof}
Parts \ref{purityflats0}-\ref{purityflats3} may be found in~\cite[Lemma 2.9]{birdwilliamsonhomological}. So it remains to prove \ref{purityflats4}. The map is continuous by \ref{purityflats3}, closed by \ref{purityflats2}, and bijective by \ref{purityflats1}, and as such is a homeomorphism. 
\end{proof}

As such, in terms of purity, there is no distinction whatsoever between the triangulated category $\T$ and the finitely accessible category $\Flat{\T^{\c}}$.

\begin{chunk}\label{definablebuilding}
	To expedite proofs, we will occasionally appeal to what we term \ti{definable building}. We shall say that a class of objects $\msf{X}\subset\T$ definably builds an object $T\in\T$ if $T\in\msf{Def}(\msf{X})$. Since
	\[
	\msf{Def}(\msf{X})=\{T\in\T:\msf{y}T\in\msf{Def}(\msf{y}\msf{X})\}
	\]
	by \cref{purityflats}\ref{purityflats3}, it is clear that $\msf{X}$ definably builds $T$ if and only if $\msf{y}\msf{X}$ definably builds $\msf{y}T$. We will also apply this to classes of objects and say that $\msf{X}$ definably builds a class $\msf{Y}$ if $\msf{Y} \subseteq \msf{Def}(\msf{X})$. In the finitely accessible category $\Mod{\T^{\c}}$, it is more elementary to give an alternative formulation for definable building: if $\mc{A}\subset \Mod{\T^{\c}}$, then $\mc{A}$ definably builds $F\in\Mod{\T^{\c}}$ if and only if there is a pure monomorphism
	\[
	F\to A \t{ where }A\in(\mc{A})^{\rlim,\msf{Prod}},
	\]
	where $(\mc{A})^{\rlim,\msf{Prod}}$ denotes the closure of $\mc{A}$ under direct limits and direct products. In particular, if $\msf{X}\subset\T$ definably builds $T\in\T$, then there is a pure monomorphism
	\[
	\msf{y}T\to X
	\]
	for some $X\in(\msf{y}\msf{X})^{\rlim,\msf{Prod}}$.
\end{chunk}

\begin{chunk}\label{definablelfp}
Suppose $\scr{A}$ and $\scr{B}$ are finitely accessible categories with products, and $\mc{D}\subseteq\scr{A}$ and $\mc{D}'\subseteq\scr{B}$ are definable subcategories. The appropriate notion of a purity-preserving functor between $\mc{D}$ and $\mc{D}'$ is that of a \emph{definable functor}. A functor $F\colon\mc{D}\to\mc{D}'$ is \emph{definable} if it preserves direct limits and direct products. Information on these functors can be found at \cite[\S 13]{dac}. The pertinent point is that definable functors preserve pure exact sequences and pure injective objects. In the algebraic setting we follow, definable functors were first discussed in \cite{kredc}, while in the model theoretic setting they were considered in \cite{prestinterp}.
\end{chunk}

\subsection{Triangles and enhancements}
We finally record some miscellaneous background facts which we will occasionally use.
\begin{chunk}\label{triangulatedbackground}
By Brown representability, a triangulated functor $F\colon \T \to \U$ between compactly generated triangulated categories has a right adjoint if and only if it preserves coproducts, and has a left adjoint if and only it preserves products. If $F \dashv G$ is an adjunction, then by a simple adjunction argument one sees that $F$ preserves compacts if and only if $G$ preserves coproducts. We refer the reader to~\cite[\S 5]{Krausenotes} for more details. We emphasize that throughout the paper we avoid assuming that functors are triangulated wherever possible. In essence, this is possible precisely because the pure triangles in a triangulated category depend only on the additive structure. However, we will frequently compare results with more familiar triangulated versions, where these recollections will be used.
\end{chunk}

\begin{chunk}\label{homologycolimits}
Since triangulated categories in general do not admit colimits, it is often convenient to work with the weaker notion of homology colimits, see \cite{krcoh}. If $(X_i)_{i \in I}$ is a diagram in a compactly generated triangulated category $\T$, the homology colimit $\msf{homcolim}_I X_i$ is an object of $\T$ together with a cone $(X_i) \to \msf{homcolim}_I X_i$ such that the induced map \[\msf{colim}_I \y X_i \to \y(\msf{homcolim}_I X_i)\] is an isomorphism in $\Mod{\T^\c}$. 
Note that in a compactly generated triangulated category, any object is isomorphic to a homology colimit of a filtered system of compacts. Indeed, if $X\in\T$, we may consider $\y X\in\Flat{\T^{\c}}$. There is then a filtered system $(\y A_{i}, \y f_{ij})_{i,j}$ such that $\rlim_{I}\y A_{i}=\y X$. Since each $A_{i}$ is compact, this system descends along $\y$ to a filtered system $(A_{i}, f_{ij})$. Similarly, the colimit maps $\y A_{i}\to \y X$ are of the form $\y f_{i}$, where $f_{i}\colon A_{i}\to X$. By construction, and definition, it is clear that the object $X$ together with the maps $f_{i}$ exhibit $X$ as the homology colimit of the filtered system $(A_{i},f_{ij})$.
\end{chunk}

\begin{chunk}\label{enhancements}
Occasionally we consider the case when our compactly generated triangulated category $\T$ has an $\infty$-categorical enhancement, by which we mean, that there exists a stable $\infty$-category $\scr{C}$ such that $h\scr{C}$ is triangulated equivalent to $\T$. We note that as $\T$ is compactly generated, any such enhancement $\scr{C}$ is automatically presentable~\cite[Corollary 1.4.4.2 and Remark 1.4.4.3]{HA}. To link this back to purity, we note that if $\T$ has an enhancement, then a full subcategory $\mc{D}$ of $\T$ is definable if and only if it is closed under products, filtered homotopy colimits, and pure subobjects, or equivalently closed under products, pure subobjects, and pure quotients, see~\cite[Theorem 3.12]{laking} and~\cite[Proposition 6.8]{dualitypairs}. We say that $F\colon \T \to \U$ arises from a functor of $\infty$-categories, if there exists an $\infty$-functor between enhancements for $\T$ and $\U$ whose derived functor is $F$. 
\end{chunk}


\section{The universal property of coherent functors}\label{sec:universalprop}

In this section we introduce coherent functors into finitely accessible categories; this definition provides the appropriate notion of a purity preserving functor from a triangulated category to a finitely accessible category. The crucial result of this section, and arguably of the paper, is that the restricted Yoneda embedding is universal among all coherent functors. This enables us to extend coherent functors to definable functors of finitely accessible categories. Firstly, we formalise what we mean by a coherent functor.

\begin{defn}\label{acoherent}
Let $\scr{A}$ be a finitely accessible category with products, and $\T$ be a compactly generated triangulated category. A functor $H\colon \T \to \scr{A}$ is said to be \emph{coherent} if it preserves coproducts and products, and sends pure triangles to pure exact sequences. We denote the category of coherent functors $\T\to\scr{A}$ by $\msf{Coh}(\T, \scr{A})$.
\end{defn}

We give examples of some naturally arising coherent functors in \cref{cohexamples}. We use the term `coherent' following Krause in \cite{krcoh}, who considered these functors in the case when $\scr{A}=\ab$. The relationship between the coherent functors of \cite{krcoh} and those defined above is discussed in \cref{abcoh}.

As illustrated in \cref{purityflats}, the pure structure of a compactly generated triangulated category $\T$ is identical to that on the category $\Flat{\T^{\c}}$. In particular, one would hope, or expect, that a coherent functor $H\colon\T\to\scr{A}$ would, in some way interact well with purity preserving functors $\Flat{\T^{\c}}\to\scr{A}$. In fact, investigating this relationship provides the main result of this section. 

\begin{thm}\label{uniprop}
Let $\T$ be a compactly generated triangulated category, and let $H\colon \T \to \scr{A}$ be a functor, where $\scr{A}$ is a finitely accessible category with products. Then the following are equivalent:
\begin{enumerate}[label=(\arabic*)]
\item\label{coherentitem1} the functor $H\colon \T \to \scr{A}$ is coherent;
\item\label{coherentitem2} there exists a definable functor $\hat{H}\colon \Flat{\T^\c} \to \scr{A}$ such that $H = \hat{H} \circ \msf{y}$.
\end{enumerate}
Moreover, any such definable lift $\hat{H}\colon \Flat{\T^\c} \to \scr{A}$ is unique.
\end{thm}

\begin{rem}\label{whyflatsremark}
The reader may wonder why the lift is only defined on flat functors rather than on the whole module category. One reason for this is that the pure structure contained in the flat functors is equivalent to that in the triangulated category by \cref{purityflats}. Nonetheless, the existence of a functor on flat objects often implies the existence of an extension to the whole module category by universal property, as demonstrated in \cref{moduleextension}, also see \cref{leftexactextension} for a related statement.
\end{rem}

Before proceeding with the proof of the main result of this section, we give an alternative characterisation of coherent functors. In the case when $\scr{A} = \ab$ this appears in~\cite[Theorem A and Proposition 5.1]{krcoh}.

\begin{prop}\label{enhanceddirectlimits}
	Let $\T$ be a compactly generated triangulated category, and let $\scr{A}$ be a finitely accessible category with products. Then a functor $H\colon \T \to \scr{A}$ is coherent if and only if $H$ preserves products and sends filtered homology colimits to direct limits.
\end{prop}
\begin{proof}
	Firstly suppose that $H$ preserves products and sends filtered homology colimits to direct limits. As any pure triangle is a filtered homology colimit of split triangles, and any pure exact sequence is a direct limit of split exact sequences, $H$ preserves pure triangles. Now any coproduct may be written as a filtered homology colimit of finite coproducts by taking the colimit over the set of finite subsets filtered by inclusion. As $F$ preserves products, it preserves finite coproducts and hence $H$ preserves arbitrary coproducts.
	
	We now show the converse, so we suppose that $H$ is coherent. We will show that the natural map $\theta\colon\rlim_{I} FX_{i}\to F(\msf{homcolim}_{I}X_{i})$ is an isomorphism for any filtered system $\{X_{i}\}_{I}$. As $\scr{A}$ is finitely accessible, it suffices to show that $\Hom_{\scr{A}}(A,\theta)$ is an isomorphism for every $A\in\msf{fp}(\scr{A})$. As $H$ preserves coproducts, products and pure triangles, $\Hom_{\scr{A}}(A,H(-))\colon \T\to\ab$ is coherent, and thus sends filtered homology colimits to filtered colimits by \cite[Proposition 5.1]{krcoh}. Consequently, $\Hom_{\scr{A}}(A,H(\msf{homcolim}_{I}X_{i}))\simeq \rlim_{I}\Hom_{\U}(A,HX_{i})$, and the right hand side is moreover equivalent to $\Hom_{\scr{A}}(A,\rlim_{I}FX_{i})$, as $A$ is finitely presented. Therefore $H$ sends filtered homology colimits to direct limits as required.
\end{proof}


\subsection{The proof of \cref{uniprop}}
We now turn our attention to proving \cref{uniprop}. As the proof relies heavily on the tools of finitely accessible categories and the functors between them, before commencing we recall some constructions, terminology, and results. Most of the constructions relating to finitely accessible categories can be found in \cite{CrawleyBoevey}, while the discussions of functors between such categories rely more on \cite{Krauselfp}.

\begin{chunk}\label{appendix.D}
Let us first recall how one can embed a finitely accessible category with products $\mathscr{A}$ into a locally coherent abelian category. Given a small preadditive category $\C$, following \cite{Krauselfp}, we set $\msf{mop}(\C):=\mod{\C^{\op}}^{\op}=\msf{fp}(\C,\ab)^{\op}$. Now, given a finitely accessuble category $\scr{A}$ with products, following~\cite[(3.3)]{CrawleyBoevey} we define $\mbb{D}(\mathscr{A}) := \Flat{\msf{mop}(\msf{fp}(\mathscr{A}))}$; we use the blackboard bold font in order to distinguish it from the derived category. By \cite[(3.3)]{CrawleyBoevey} this is a locally coherent category such that $\msf{fpInj}(\mbb{D}(\scr{A}))$ is finitely accessible.

There is a fully faithful functor $\msf{d}_\scr{A}\colon \scr{A} \to \mbb{D}(\scr{A})$ which preserves products and direct limits (hence is definable) which is defined by universal property: it is the unique functor commuting with direct limits making the diagram
\[\begin{tikzcd}
\msf{fp}(\scr{A})\arrow[r, "\msf{h}"] \arrow[d, hookrightarrow] & \msf{mop}(\msf{fp}(\scr{A})) \arrow[d, "\msf{y}"] \\
\scr{A} \arrow[r, "\msf{d}_{\scr{A}}"'] & \mbb{D}(\scr{A})	
\end{tikzcd}\]
commute. Moreover, $\msf{d}_\scr{A}$ induces an equivalence $\msf{d}_\scr{A}\colon \scr{A} \xrightarrow{\sim} \msf{fpInj}(\mbb{D}(\scr{A}))$, and the assignment $\scr{A} \mapsto \mbb{D}(\scr{A})$ gives, by \cite[Corollary 9.8]{Krauselfp}, a bijection between finitely accessible categories with products and locally coherent categories with enough fp-injectives.
\end{chunk}

\begin{chunk}\label{extodef}
Let $\scr{A}$ and $\scr{B}$ be finitely accessible categories with products, and let $G\colon \scr{A} \to \scr{B}$ be a definable functor in the sense of \cref{definablelfp}. From $G$, one may construct an exact functor $\epsilon(G)\colon \msf{mop}(\msf{fp}(\scr{B})) \to \msf{mop}(\msf{fp}(\scr{A}))$ as follows. By applying~\cite[Universal Property 10.4, Corollary 10.5]{Krauselfp} we obtain a functor $G^*\colon \mbb{D}(\scr{A}) \to \mbb{D}(\scr{B})$ such that $G^*d_\scr{A} = d_\scr{B}G$, and since $G$ is definable, $G^*$ has a left adjoint $G_!$ which is exact and preserves finitely presented objects. As such we obtain $G_!\vert\colon \msf{fp}(\mbb{D}(\scr{B})) \to \msf{fp}(\mbb{D}(\scr{A}))$. 

Now by definition of $\mbb{D}(\scr{A})$ (cf. \cref{appendix.D}), we have $\msf{fp}(\mbb{D}(\scr{A})) = \flat{\msf{mop}(\msf{fp}(\scr{A}))}$, and therefore the Yoneda functor $\y\colon \msf{mop}(\msf{fp}(\scr{A})) \to \msf{fp}(\mbb{D}(\scr{A}))$ is an equivalence. From this, we define $\epsilon(G)$ to be the unique functor making
\[\begin{tikzcd}
\msf{mop}(\msf{fp}(\scr{B})) \ar[r, "\epsilon(G)"] \ar[d, "\y"', "\simeq"] & \msf{mop}(\msf{fp}(\scr{A})) \ar[d, "\y", "\simeq"'] \\
\msf{fp}(\mbb{D}(\scr{B})) \ar[r, "G_!\vert"'] & \msf{fp}(\mbb{D}(\scr{A}))
\end{tikzcd}\]
commute. 

Conversely, given an exact functor $\Phi\colon\msf{mop}(\msf{fp}(\scr{B})) \to \msf{mop}(\msf{fp}(\scr{A}))$ we now construct a definable functor $\delta(\Phi)\colon \scr{A} \to \scr{B}$. As in the above diagram, $\Phi$ gives a functor $\tilde{\Phi}\colon \msf{fp}(\mbb{D}(\scr{B})) \to \msf{fp}(\mbb{D}(\scr{A}))$ and by~\cite[Theorem 10.1]{Krauselfp} this induces an exact, direct limit preserving functor $\tilde{\Phi}^*\colon \mbb{D}(\scr{B}) \to \mbb{D}(\scr{A})$ which has a right adjoint $\tilde{\Phi}_*$. Moreover, this right adjoint $\tilde{\Phi}_*$ preserves direct limits and fp-injective objects. As $d_\scr{A}\colon \scr{A} \to \msf{fpInj}(\mbb{D}(\scr{A}))$ is an equivalence by \cref{appendix.D}, we may define $\delta(\Phi)$ to be the unique functor making 
\[\begin{tikzcd}
\scr{A} \ar[r, "\delta(\Phi)"] \ar[d, "d_\scr{A}"', "\simeq"] & \scr{B} \ar[d, "d_\scr{B}", "\simeq"'] \\
\msf{fpInj}(\mbb{D}(\scr{A})) \ar[r, "\tilde{\Phi}_*\vert"'] & \msf{fpInj}(\mbb{D}(\scr{B}))
\end{tikzcd}\]
commute. Combining these constructions, one obtains the following theorem:
\end{chunk}

\begin{thm}[{\cite[Theorem 11.2]{Krauselfp}}]\label{thm:deftoex}
Let $\scr{A}$ and $\scr{B}$ be finitely accessible categories with products. The assignments $\epsilon(-)$ and $\delta(-)$ from \cref{extodef} give an equivalence of categories
\[\msf{Def}(\scr{A}, \scr{B})^\op \simeq \msf{Ex}(\msf{mop}(\msf{fp}(\scr{B})), \msf{mop}(\msf{fp}(\scr{A})))\]
between the opposite of the category of definable functors from $\scr{A}$ to $\scr{B}$, and the category of exact functors from $\msf{mop}(\msf{fp}(\scr{B}))$ to $\msf{mop}(\msf{fp}(\scr{A}))$.
\end{thm}

\begin{chunk}
We may now turn towards proving the main result of this section. The proof of the implication \ref{coherentitem1} $\Rightarrow$ \ref{coherentitem2} of \cref{uniprop} is somewhat lengthy, so we split it up into several parts, some of which may be of independent interest.

There are, essentially, two stages to the proof. Firstly, given a coherent functor $H\colon \T \to \scr{A}$, we construct a definable functor $\hat{H}\colon \Flat{\T^\c} \to \scr{A}$; this proves to be relatively straightforward. The second stage is in showing that $\hat{H}$ interacts in the desired way with the restricted Yoneda embeddings, which is more involved.
 
The first step in the construction of the definable functor $\hat{H}$ is in passing from the triangulated world, to the realm of finitely accessible categories, where we have a vast array of non-triangulated categorical machinery at our disposal. Our first task is to formalise this passage. In order to do this, we firstly set some notation. 
\end{chunk}

\begin{chunk}\label{upanddownfinitelyaccessible}
Given a finitely presented functor $f\colon \msf{fp}(\scr{A}) \to \ab$, we denote by $f\up_\scr{A}$ the unique direct limit preserving extension to a functor $\scr{A} \to \ab$. We note that equivalently, if $f$ has presentation
\[
\Hom_{\scr{A}}(B,-)\vert_{\msf{fp}(\scr{A})}\to \Hom_{\scr{A}}(A,-)\vert_{\msf{fp}(\scr{A})}\to f\to 0,
\]
then $f\up_\scr{A}$ has presentation
\[
\Hom_{\scr{A}}(B,-)\to\Hom_{\scr{A}}(A,-)\to f\up_\scr{A} \to 0
\]
in $(\scr{A}, \ab)$. As such, $f\up_\scr{A}$ preserves products. Moreover, as $f\up_\scr{A}$ preserves direct limits, it preserves pure exact sequences since these are direct limits of split exact sequences.
\end{chunk}
\begin{lem}\label{definabletoexact}
Let $H\colon\T\to\scr{A}$ be a coherent functor. The assignment \[H\mapsto \left((f\colon \msf{fp}(\scr{A}) \to \ab) \mapsto (f\up_\scr{A} \circ H\colon \T \to \ab)\right)\] induces a functor $\msf{Coh}(\T,\scr{A})\to\msf{Ex}(\mod{\msf{fp}(\scr{A})^\op},\msf{Coh}(\T,\ab))$. 
\end{lem}
\begin{proof}
By the assumed preservation properties of $H$, we see that the composition $f\up_\scr{A}\circ H\colon\T\to\ab$ preserves coproducts, products and sends pure triangles to pure exact sequences. As such, $f\up_\scr{A} \circ H$ is a coherent functor, and consequently $H\mapsto (f\mapsto f\up_\scr{A}\circ H)$ yields a functor $\msf{Coh}(\T,\scr{A})\to(\mod{\msf{fp}(\scr{A})^\op},\msf{Coh}(\T,\ab))$. To conclude the proof, we must show that this functor takes values in exact, not just additive functors. To do this, suppose that $0\to f\to f'\to f''\to 0$ is an exact sequence in $\mod{\msf{fp}(\scr{A})^\op}$. As exactness is determined pointwise, it follows that $0\to f\up_\scr{A}H(t)\to f'\up_\scr{A}H(t)\to f''\up_\scr{A}H(t)\to 0$ is exact for all $t\in \T$, as required.
\end{proof}

The previous lemma moves the task of constructing $\hat{H}$ completely into the realm of finitely accessible categories. In order to do this, we recall two rather trivial equivalences of categories.

\begin{chunk}\label{upanddown}
The first is the equivalence between $\msf{fp}(\T^{\c},\ab)$ and $\msf{Coh}(\T,\ab)$, as mentioned in \cref{coherent}, for which we now introduce some notation. In a similar way to \cref{upanddownfinitelyaccessible}, if $f\in\msf{fp}(\T^{\c},\ab)$ has presentation
\[
\Hom_{\T}(B,-)\vert_{\T^{\c}}\to \Hom_{\T}(A,-)\vert_{\T^{\c}}\to f\to 0,
\]
we shall let $f\up_\T$ denote the corresponding coherent functor which has presentation
\[
\Hom_{\T}(B,-)\to\Hom_{\T}(A,-)\to f\up_\T \to 0
\]
in $(\T^{\c},\ab)$. On the other hand, if $G\in\msf{Coh}(\T,\ab)$, then we let $G\down_\T\in\msf{fp}(\T^{\c},\ab)$ denote the restriction of $G$ to $\T^{\c}$. Clearly $\up_\T$ and $\down_\T$ are mutually inverse functors.
\end{chunk}

\begin{chunk}\label{littleflats}
The second equivalence we need is induced by the equivalence $\msf{y}\colon\T^{\c}\xrightarrow{\sim}\flat{\T^{\c}}$ which we recalled in \cref{flatequivalences}. This gives rise to an equivalence of categories
\[ 
\begin{tikzcd}
(\T^{\c},\ab) \arrow[r, shift left = 1.1ex, "\beta"] \arrow[r, phantom, "\simeq"] \arrow[r, leftarrow, shift right = 1.1ex, swap, "\alpha"] & (\flat{\T^{\c}},\ab)
\end{tikzcd}
\]
where $\alpha\colon g\mapsto g\circ \msf{y}$, 
for every $g\in(\flat{\T^{\c}},\ab)$, and $\beta\colon f\mapsto (\msf{y}A\mapsto fA)$ for $f\in(\T^{\c},\ab)$, and all $A\in\T^{\c}$. We note that $\alpha$ and $\beta$ restrict to give an equivalence $\msf{fp}(\T^\c, \ab) \simeq \msf{fp}(\flat{\T^\c},\ab)$.
\end{chunk}

The following result is the final step of the first stage, enabling us to obtain a definable functor $\Flat{\T^{\c}}\to\scr{A}$.
\begin{prop}\label{exactequivdefinable}
Let $\T$ be a compactly generated triangulated category, and $\scr{A}$ be a finitely accessible category with products. Then there is an equivalence of categories
\[
\msf{Def}(\Flat{\T^{\c}},\scr{A}) \simeq \msf{Ex}(\mod{\msf{fp}(\scr{A})^\op}, \msf{Coh}(\T,\ab)).
\]
\end{prop}
\begin{proof}
There are equivalences of categories
\begin{align*}
\msf{Ex}(\mod{\msf{fp}(\scr{A})^\op},\msf{Coh}(\T,\ab)) &\simeq \msf{Ex}(\mod{\msf{fp}(\scr{A})^\op},\msf{fp}(\T^{\c},\ab)) & \t{induced by $\downarrow_\T$ as in \cref{upanddown},} \\
{}&\simeq \msf{Ex}(\mod{\msf{fp}(\scr{A})^\op},\msf{fp}(\flat{\T^{\c}},\ab))&\t{induced by $\beta$ as in \cref{littleflats}},\\
{}&= \msf{Ex}(\msf{mop}(\msf{fp}(\scr{A}))^\op,\msf{mop}(\flat{\T^{\c}})^\op)
&\t{by definition of $\msf{mop}$ (see \cref{appendix.D}),}\\
&\simeq \msf{Ex}(\msf{mop}(\msf{fp}(\scr{A})),\msf{mop}(\flat{\T^{\c}}))^\op &\t{by taking opposites,} \\
{} &\simeq \msf{Def}(\Flat{\T^{\c}},\scr{A}) &\t{by \cref{thm:deftoex}}.
\end{align*}
\end{proof}

\begin{chunk}\label{Fhatconstruction}
Combining \cref{definabletoexact} and \cref{exactequivdefinable}, we see that there is a functor
\[
\msf{Coh}(\T,\scr{A})\to\msf{Def}(\Flat{\T^{\c}},\scr{A}).
\]
The image of any coherent functor $H\colon\T\to\scr{A}$ under this functor provides a candidate definable functor $\hat{H}\colon\Flat{\T^{\c}}\to\scr{A}$. However, it remains to show that $\hat{H}$ is compatible with $H$ under the Yoneda embedding. In order to do this we explicitly unravel the equivalences of categories given in the proof of \cref{exactequivdefinable}.
\end{chunk} 

\begin{chunk}\label{step1}
First of all, we may combine the first three isomorphisms of \cref{exactequivdefinable} to obtain from the functor $-\up_\scr{A} \circ H\in \msf{Ex}(\mod{\msf{fp}(\scr{A})^\op}, \msf{Coh}(\T,\ab))$ of \cref{definabletoexact} a unique exact functor $\varphi\colon \msf{mop}(\msf{fp}(\scr{A}))^\op \to \msf{mop}(\msf{flat}(\T^{\c}))^\op$, which is defined by the following commutative diagram
\[
\begin{tikzcd}[column sep=1in, row sep=0.5in]
\mod{\msf{fp}(\scr{A})^\op} \arrow[r, "-\up_\scr{A}\circ H"] \arrow[d, equal] & \msf{Coh}(\T,\ab) \arrow[d, shift left = 1.5ex, "\beta_{\T}\circ\down_{\T}"] \arrow[d, phantom, "\simeq" description] \\
\msf{mop}(\msf{fp}(\scr{A}))^\op \arrow[r, "\varphi"'] & \msf{mop}(\msf{flat}(\T^{\c}))^\op \arrow[u,shift left = 1.5ex, "\up_{\T}\circ\alpha_{\T}"].
\end{tikzcd}
\]
Any $g \in \msf{mop}(\msf{fp}(\scr{A}))^{\op}$ has a presentation $\msf{h}A \to \msf{h}A' \to g \to 0$ where $A, A' \in \msf{fp}(\scr{A})$. Since $\varphi$ is exact, to understand $\varphi(g)$ it suffices to understand $\varphi(\msf{h}A)\colon \flat{\T^\c} \to \ab$ for $A \in \msf{fp}(\scr{A})$. But, as any object in $\flat{\T^\c}$ is of the form $\msf{y}X$ for some $X \in \T^\c$, it is sufficient to consider $\varphi(\msf{h}A)(\msf{y}X)$, as we now do.
\end{chunk}

\begin{lem}\label{evalonreps}
For any $A \in \msf{fp}(\scr{A})$ and $X \in \T^\c$ we have
	\[
	\varphi(\msf{h}A)(\msf{y}X)=\Hom_\scr{A}(A, HX). 
	\]
\end{lem}
\begin{proof}
By definition, $\varphi(\msf{h}A) = \beta_\T[((\msf{h}A)\up_\scr{A} \circ H)\down_\T] = \beta_\T[\Hom_\scr{A}(A, H(-))\vert_{\T^\c}]$. The claim then follows as $\beta_\T(\Hom_\scr{A}(A, H(-)))(\msf{y}X) = \Hom_\scr{A}(A, HX)$ by definition of $\beta_\T$.
\end{proof}

\begin{chunk}\label{step2}
We next turn to showing how, from the functor $\varphi^\op\colon \msf{mop}(\msf{fp}(\scr{A})) \to \msf{mop}(\flat{\T^\c})$ of \cref{step1}, we obtain $\hat{H}\colon \Flat{\T^\c} \to \scr{A}$, as in the last step of the proof of \cref{exactequivdefinable}. To simplify notation (and for ease in comparing with \cref{extodef}) we write $\Phi$ for $\varphi^\op$. The functor $\hat{H}$ is obtained by applying the construction of \cref{extodef}: namely, $\hat{H} = \delta(\Phi)$ in that notation. Diagrammatically, $\hat{H}$ arises as follows, where we write $\msf{d}_\T$ for $\msf{d}_{\Flat{\T^\c}}$.
\[
\begin{tikzcd}[column sep=0.5in, row sep=0.7in]
\msf{mop}(\msf{fp}(\scr{A})) \arrow[r, "\msf{y}", "\simeq"'] \arrow[d, "\Phi"'] & \msf{fp}(\mbb{D}(\scr{A})) \arrow[r, hookrightarrow] \ar[d, "\widetilde{\Phi}"] &
\mbb{D}(\scr{A}) \arrow[d, xshift=-1.7mm, "\tilde{\Phi}^*"'] \arrow[d, phantom, "\dashv"] \arrow[r, hookleftarrow] & 
\msf{fpInj}(\mbb{D}(\scr{A})) & 
\scr{A} \arrow[l, "\msf{d}_\scr{A}"', "\simeq"] \\
\msf{mop}(\msf{flat}(\T^\c)) \arrow[r, "\msf{y}"', "\simeq"] & \msf{fp}(\mbb{D}(\msf{Flat}(\T^\c))) \arrow[r, hookrightarrow] &
\mbb{D}(\msf{Flat}(\T^\c)) \arrow[u, xshift=1.7mm, "\tilde{\Phi}_*"']  \arrow[r, hookleftarrow] & 
\msf{fpInj}(\mbb{D}(\msf{Flat}(\T^\c))) \arrow[u, "\tilde{\Phi}_*\mid"']& 
\msf{Flat}(\T^\c). \arrow[l, "\msf{d}_\T"', "\simeq"] \arrow[u, "\hat{H}"']
\end{tikzcd}
\]

By general abstract nonsense, one may identify $\widetilde{\Phi}_*$ in a more explicit form, as the next lemma shows.
\end{chunk}

\begin{lem}\label{lowerstar}
Let $\tilde{\Phi}_{*}\colon\mbb{D}(\msf{Flat}(\T^{\c}))\to \mbb{D}(\scr{A})$ be the functor constructed above in \cref{step2}. Then $\tilde{\Phi}_{*} = - \circ \Phi$.
\end{lem}
\begin{proof}
Consider the functor $\Phi\colon\msf{mop}(\msf{fp}(\scr{A}))\to \msf{mop}(\msf{flat}(\T^{\c}))$. By~\cite[\S 4]{Krauselfp} this gives rise to an adjunction 
\[
\begin{tikzcd}
\Mod{\msf{mop}(\msf{flat}(\T^{\c}))} \arrow[r, "R"', yshift=-1mm]& \Mod{\msf{mop}(\msf{fp}(\scr{A}))}. \arrow[l, "L"', yshift=1mm]
\end{tikzcd}
\]
where $R = -\circ \Phi$ and $L$ is the unique colimit-preserving functor which sends $\Hom(-,X)$ to $\Hom(-,\Phi(X))$. Recall that $\mbb{D}(\scr{A}) = \msf{Flat}(\msf{mop}(\msf{fp}(\scr{A})))$ by definition (see \cref{appendix.D}). We now show that both $L$ and $R$ preserve flats and hence restrict to an adjunction 
\[
\begin{tikzcd}
\mbb{D}(\msf{Flat}(\T^{\c})) \arrow[r, "R"', yshift=-1mm]& \mbb{D}(\msf{Flat}(\U^{\c})). \arrow[l, "L"', yshift=1mm]
\end{tikzcd}
\]
Since $L$ preserves colimits and sends $\Hom(-,X)$ to $\Hom(-, \Phi(X))$, it preserves flats, as any flat is a colimit of representables. The functor $R$ preserves flats by~\cite[Lemma 6.5]{Krauselfp} as $\Phi$ is exact. 

By construction, $\widetilde{\Phi}^*$ is the unique direct limit preserving functor $\mbb{D}(\scr{A}) \to \mbb{D}(\Flat{\T^\c})$ such that $\msf{y}\circ \Phi = \widetilde{\Phi}^* \circ \msf{y}$, see \cref{extodef}. Therefore, by uniqueness of $\widetilde{\Phi}^*$, we see that $\widetilde{\Phi}^* = L$, and hence by uniqueness of adjoints, $\widetilde{\Phi}_* = R$ as required.
\end{proof}

We record two final auxiliary lemmas before we can prove the compatibility with Yoneda embeddings.
\begin{lem}\label{calculated}
Let $\scr{A}$ be a finitely accessible category with products. Then for any $M \in \scr{A}$ and $A \in \msf{fp}(\scr{A})$ we have $(\msf{d}_{\scr{A}}M)(\msf{h}A) = \Hom_\scr{A}(A,M)$.
\end{lem}
\begin{proof}
We may write $M = \rlim M_i$ where each $M_i \in \msf{fp}(\scr{A})$, and therefore as $\msf{d}_\scr{A}$ preserves direct limits, we have $\msf{d}_\scr{A} M = \rlim \msf{d}_\scr{A} M_i = \rlim \msf{y}\msf{h}M_i$ where the last equality follows from \cref{appendix.D}. Therefore 
\begin{align*}
(\msf{d}_\scr{A} M)(\msf{h}A) &= \rlim[(\msf{y}\msf{h}M_i)(\msf{h}A)]  \\ 
&= \rlim\Hom_{\msf{mop}(\msf{fp}(\scr{A}))}(\msf{h}A, \msf{h}M_i)\\ 
&= \rlim\Hom_{\msf{fp}(\scr{A})}(A, M_i)\\ 
&= \Hom_\scr{A}(A, M) 
\end{align*}
as required.
\end{proof}

\begin{lem}\label{applyd}
Let $M,N \in \scr{A}$. Then $M = N$ if and only if $(\msf{d}_\scr{A} M)(\msf{h}A) = (\msf{d}_\scr{A} N)(\msf{h}A)$ for each $A \in \msf{fp}(\scr{A})$.
\end{lem}
\begin{proof}
The forward implication is clear so we only prove the converse. As $\scr{A}$ is finitely accessible, the converse follows from \cref{calculated} by writing any object of $\scr{A}$ as a direct limit of finitely presented objects.
\end{proof}
We now show that $\hat{H}$ commutes with the restricted Yoneda embedding in the desired way.

\begin{prop}\label{commutesoncompacts}
If $X$ is an object of $\T$, then $HX = \hat{H}\msf{y}X$.
\end{prop}
\begin{proof}
Since $\hat{H} = \msf{d}_\scr{A}^{-1}\tilde{\Phi}_*\msf{d}_\T$ by definition, we must show that $HX = \msf{d}_\scr{A}^{-1}\tilde{\Phi}_*\msf{d}_\T \msf{y} X$. By \cref{applyd}, it therefore suffices to show that
\begin{equation}\label{yonedcommutents}
[\msf{d}_\scr{A}(HX)](\msf{h}A) = [\tilde{\Phi}_*(\msf{d}_\T\msf{y}X)](\msf{h}A)
\end{equation}
for each $A \in \msf{fp}(\scr{A})$. By \cref{calculated}, the left hand side of \cref{yonedcommutents} is equal to $\Hom_{\scr{A}}(A, HX)$.

By \cref{appendix.D} we have $\msf{d}_\msf{T}\msf{y}X = \msf{y}\msf{h}\msf{y}X$, whenever $X$ is compact. If $X$ is arbitrary, we may, write $X=\msf{homcolim}_{I}C_{i}$ as a filtered homology colimit of compact objects, see \cref{homologycolimits}. Since $\y(\msf{homcolim}_{I}C_{i})\simeq \rlim_{I}\y C_{i}$, see \cref{homologycolimits}, we have isomorphisms 
\begin{align*}
\msf{d}_{\T}\y X& \simeq \msf{d}_{\T}\rlim\y C_{i} & \\
&\simeq \rlim\msf{d}_{\T}\y C_{i} &\t{as $\msf{d}_{\T}$ preserves direct limits},\\
&\simeq\rlim\y\msf{h}\y C_{i}  &\t{as each $C_{i}$ is compact}.
\end{align*}
Thus, by incorporating \cref{lowerstar} and noting that $\tilde{\Phi}_*$ commutes with colimits, we see that 
\[
[\tilde{\Phi}_*(\msf{d}_\T\y X)](\msf{h}A)\simeq \rlim_{I} [\tilde{\Phi}_*(\msf{d}_\T\y C_{i})](\msf{h}A)\simeq \rlim_{I}(\msf{y}\msf{h}\msf{y}C_{i} \circ \phi)(\msf{h}A).
\]

Thus for \cref{yonedcommutents} to hold, we must show that $\rlim_{I}(\msf{y}\msf{h}\msf{y}C_{i} \circ \phi)(\msf{h}A) = \Hom_{\scr{A}}(A,HX).$

By definition of $\msf{y}$ and the fact that $\msf{mop}(\msf{flat}(\T^\c))$ is a full subcategory of $(\msf{flat}(\T^\c), \ab)^\op$, we see that for every $i$ we have
\[(\msf{y}\msf{h}\msf{y}C_{i} \circ \phi)(\msf{h}A) = \Hom_{\msf{mop}(\msf{flat}(\T^\c))}(\phi(\msf{h}A), \msf{h}\msf{y}C_{i}) = \Hom_{(\msf{flat}(\T^\c),\ab)}(\msf{h}\msf{y}C_{i}, \phi(\msf{h}A)).\]
By using the Yoneda lemma in conjunction with \cref{evalonreps}, this is moreover equal to 
\[\phi(\msf{h}A)(\msf{y}C_{i}) = \Hom_\scr{A}(A,HC_{i}).\]
Since $A\in\msf{fp}(\scr{A})$, it follows that we have
\[
\rlim_{I}(\msf{y}\msf{h}\msf{y}C_{i} \circ \phi)(\msf{h}A) \simeq \rlim_{I}\Hom_{\scr{A}}(A,H C_{i})\simeq \Hom_{\scr{A}}(A,\rlim_{I} H C_{i}),
\]
but, as $H$ is coherent, it follows that this is equivalent to $\Hom_{\scr{A}}(A,HX)$, as desired.
\end{proof}

Combining all of the above, we now have the necessary tools to prove \cref{uniprop}.
\begin{proof}[Proof of \cref{uniprop}]
	For the implication \ref{coherentitem1} implies \ref{coherentitem2}, the functor $\hat{H}$ is constructed in \cref{Fhatconstruction}, and has the claimed commutation property with Yoneda by \cref{commutesoncompacts}. The converse is clear from the fact that $\hat{H}$ is definable and $\y$ is coherent. 
	
	So it remains to prove the uniqueness claim. As such, we suppose that $G\colon \Flat{\T^\c} \to \scr{A}$ is definable and satisfies $G \circ \y = H$. We will construct a natural isomorphism $\theta\colon \hat{H} \Rightarrow G$. Let $X\in\Flat{\T^{\c}}$ and consider the first two terms of a pure injective resolution of $X$ as follows. Take the pure injective hull $PE(X)$ of $X$, and extend this to a pure exact sequence \[0 \to X \to PE(X) \to Q \to 0.\] 
	Note that as $\Flat{\T^{\c}}$ is definable, each term of this sequence is in $\Flat{\T^{\c}}$.
	Repeating on $Q$, we get another pure exact sequence \[0 \to Q \to PE(Q) \to Q' \to 0.\]
	Since $\hat{H}$ and $G$ are definable, they preserve pure exact sequences, see \cref{definablelfp}, and so by splicing we obtain an exact sequence
	\[0 \to \hat{H}X \to \hat{H}(PE(X)) \xrightarrow{\hat{H}p} \hat{H}(PE(Q))\]
	and similarly for $G$, where $p\colon PE(X) \to PE(Q)$ is a morphism in $\msf{Pinj}(\T)$. 

As $\msf{y}\colon \msf{Pinj}(\T) \xrightarrow{\sim} \msf{Pinj}(\Flat{\T^\c})$ is an equivalence of categories by~\cref{purityflats}\ref{purityflats1}, there exists a map $f\colon E_0 \to E_1$ in $\msf{Pinj}(\T)$ such that $\msf{y}f = p$. Therefore there is a commutative diagram with exact rows
\[
\begin{tikzcd}
0 \ar[r] & \hat{H}X \ar[r] & \hat{H}\msf{y}E_0 \ar[r, "\hat{H}\msf{y}f"] \ar[d, "\simeq"] & \hat{H}\msf{y}E_1 \ar[d, "\simeq"] \\
0 \ar[r] & GX \ar[r] & G\msf{y}E_0 \ar[r, "G\msf{y}f"'] & G\msf{y}E_1
\end{tikzcd}
\]
where the vertical maps are isomorphisms since $G\y X=HX=\hat{H}\y X$. Consequently, we obtain a natural map $\theta_X\colon \hat{H}X \to GX$ which is an isomorphism by construction. 
\end{proof}


\subsection{Extension to all modules}\label{sec:extending}
The universal property proved in the previous section provides a factorisation of any coherent functor through the category of flat functors on $\T^\c$, but sometimes it is convenient to extend this to the whole module category. We give such an extension in this subsection, and then relate it to other universal properties proved by Beligiannis~\cite{bel} and Krause~\cite{krsmash}.

\begin{thm}\label{moduleextension}
Let $\T$ be a compactly generated triangulated category, and let $H\colon \T \to \scr{A}$ be a coherent functor, where $\scr{A}$ is a finitely accessible category with products and kernels. Then there exists a unique left exact, definable functor $\bar{H}\colon \Mod{\T^\c} \to \scr{A}$ such that $H = \bar{H} \circ \msf{y}$.
\end{thm}
\begin{proof}
By \cref{uniprop}, we have a unique definable functor $\hat{H}\colon \Flat{\T^\c} \to \scr{A}$ such that $H = \hat{H} \circ \msf{y}$. We have $\msf{fpInj}(\Mod{\T^\c}) = \Flat{\T^\c}$ by \cref{flatequivalences}, and therefore $\mathbb{D}(\Flat{\T^\c}) = \Mod{\T^\c}$ by \cref{appendix.D}. As such, there exists a unique left exact, definable functor $\bar{H}\colon \Mod{\T^\c} \to \scr{A}$ extending $\hat{H}$ by \cite[Universal Property 10.4]{Krauselfp}; we note that the hypotheses for this are satisfied, since by \cite[\S2.3, Corollary]{CrawleyBoevey}, any finitely accessible category with kernels has left exact direct limits.
\end{proof}

In the circumstances of the preceding theorem, we actually obtain a left adjoint to $\bar{H}$ for free.

\begin{cor}\label{moduleextensionleftadjoint}
Let $H\colon \T \to \scr{A}$ be a coherent functor, where $\scr{A}$ is a finitely accessible category with products and kernels. Then the functor $\bar{H}\colon\Mod{\T^{\c}}\to\scr{A}$ admits a left adjoint which preserves finitely presented objects.
\end{cor}

\begin{proof}
The existence of the left adjoint follows from \cite[Universal Property 10.4]{Krauselfp}, and it preserves finitely presented objects since $\bar{H}$ preserves direct limits.
\end{proof}

Certain preservation properties of $\hat{H}$ are equivalent to the same property for $\bar{H}$.

\begin{lem}\label{preservesfpsequiv}
Let $H\colon\T\to\scr{A}$ be a coherent functor, where $\scr{A}$ is a finitely accessible category with products and kernels, and suppose that $\msf{fp}(\scr{A})$ is closed under subobjects. Then the following are equivalent:
\begin{enumerate}[label=(\arabic*)]
\item\label{fpequiv1} $H$ sends compact objects to finitely presented objects;
\item\label{fpequiv2} $\hat{H}\colon\Flat{\T^{\c}}\to\scr{A}$ preserves finitely presented objects;
\item\label{fpequiv3} $\bar{H}\colon\Mod{\T^{\c}}\to\scr{A}$ preserves finitely presented objects.
\end{enumerate}
\end{lem}
\begin{proof}
The equivalence of \ref{fpequiv1} and \ref{fpequiv2} is trivial, as $\y\colon\T^{\c}\xrightarrow{\sim}\flat{\T^{\c}}$ is an equivalence. Note that we do not require $\scr{A}$ to have kernels, or for $\msf{fp}(\scr{A})$ to be closed under subobjects for this equivalence to hold. That \ref{fpequiv3} implies \ref{fpequiv2} is also trivial, as $\bar{H}$ agrees with $\hat{H}$ on flat functors.

Finally, for \ref{fpequiv2} implies \ref{fpequiv3}, the category $\mod{\T^{\c}}$ is Frobenius (see, for instance, \cite[Lemma 5.2]{krcq}), with projective-injective objects $\msf{flat}(\T^{\c})$. Consequently, for any $f\in\mod{\T^{\c}}$ there is a monomorphism $0\to f\to\y A$, for some $A\in\T^{\c}$. As $\bar{H}$ has a left adjoint by \cref{moduleextensionleftadjoint}, it is left exact, and thus $0\to\bar{H}f\to\bar{H}\y A$ is a monomorphism in $\scr{A}$. But $\bar{H}\y A=\hat{H}\y A$, which is finitely presented by assumption. As $\msf{fp}(\scr{A})$ is closed under subobjects, we see that $\bar{H}f\in\msf{fp}(\scr{A})$, as desired.  
\end{proof}

We now compare $\bar{H}$ with the results of~\cite{bel} and~\cite{krsmash}. In order to do so, we first investigate when $\bar{H}$ is exact. 

\begin{prop}\label{cohomologicaliffexact}
Let $H\colon \T \to \scr{A}$ be a coherent functor, where $\scr{A}$ is a finitely accessible category with products and kernels. Consider the induced definable functor $\bar{H}\colon\Mod{\T^{\c}}\to\scr{A}$ of \cref{moduleextension}. Then $\bar{H}$ is exact if and only if $H\colon\T\to\scr{A}$ is homological.
\end{prop}

\begin{proof}
First, suppose that $\bar{H}$ is exact. As $\y$ is homological and $H = \bar{H} \circ \y$, it is clear that $H$ is homological as required. For the converse, we use an argument dual to that of \cite[Lemma 2.1]{krsmash}. Suppose that $0\to G_{1}\to G_{2}\to G_{3}\to 0$ is an exact sequence in $\Mod{\T^{\c}}$. Consider the start of an injective resolution of $G_{1}$:
\[
0\to G_{1}\to \y X_{1}^{0}\xrightarrow{\y \alpha_{1}} \y X_{1}^{1} \xrightarrow{\y \beta_{1}} \y X_{1}^{2},
\]
where the resolution takes this form by the equivalence of categories $\y\colon \msf{Pinj}(\T) \xrightarrow{\sim} \msf{Inj}(\T^\c)$. There is then a triangle $\Sigma^{-1}X_{1}^{2}\to X_{1}^{0} \xrightarrow{\alpha_{1}}X_{1}^{1} \xrightarrow{\beta_{1}}X_{1}^{2}$, and thus an exact sequence \[H\Sigma^{-1}X_{1}^{2}\to H X_{1}^{0} \xrightarrow{H \alpha_{1}}H X_{1}^{1} \xrightarrow{H \beta_{1}}H X_{1}^{2}\] since $H$ is assumed to be homological. Since $\bar{H}$ is left exact and $\bar{H}\y=H$, it follows that 
\[
0\to \bar{H}G_{1}\to  HX_{1}^{0} \xrightarrow{H\alpha_{1}}H X_{1}^{1} \xrightarrow{H \beta_{1}}H X_{1}^{2}
\]
is an exact sequence. A similar argument for $G_{3}$ and an application of the horseshoe lemma gives us a commutative diagram
\[
\begin{tikzcd}[column sep=1.6cm, row sep=0.5cm]
{} & 0 \arrow[d] & 0\arrow[d]& 0\arrow[d] & 0\arrow[d] \\
0 \arrow[r]& \bar{H}G_{1}\arrow[d]\arrow[r] & H X_{1}^{0}\arrow[d] \arrow[r, "H \alpha_{1}"] & H X_{1}^{1} \arrow[d]\arrow[r, "H\beta_{1}"]& H X_{1}^{2}\arrow[d]\\
0 \arrow[r]& \bar{H}G_{2} \arrow[d]\arrow[r]& H (X_{1}^{0}\oplus X_{3}^{0})\arrow[d]\arrow[r, "H (\alpha_{1}\oplus \alpha_{2})"]& H (X_{1}^{1}\oplus X_{3}^{1}) \arrow[d]\arrow[r, "H (\beta_{1}\oplus \beta_{2})"]& H (X_{1}^{2}\oplus X_{3}^{2})\arrow[d]\\
0\arrow[r] & \bar{H}G_{3} \arrow[d]\arrow[r]& H X_{3}^{0}\arrow[d]\arrow[r, "H \alpha_{2}"] & H X_{3}^{1} \arrow[d]\arrow[r, "H \beta_{2}"]& H X_{3}^{2}\arrow[d]\\
{} & 0 & 0 & 0 & 0
\end{tikzcd}
\]
where, a priori, all rows and the second to fourth columns are exact. Diagram chasing shows that the first column is also exact, as required.
\end{proof}

\begin{chunk}
We may now use \cref{cohomologicaliffexact} to relate the functor obtained in \cref{moduleextension} to other extensions of functors to modules which exist in the literature.

The first such statement we consider is Krause~\cite[Corollary 2.4]{krsmash}, which states that any coproduct preserving, homological functor $H\colon\T\to\A$, where $\A$ is an AB5 category, extends uniquely to an exact coproduct preserving functor $H'\colon \Mod{\T^{\c}}\to\A$ which satisfies $H=H'\circ\y$. 

Another notable result is due to Beligiannis~\cite[Theorem 3.4]{bel}. This states that if $H\colon\T\to\A$, where $\A$ is an arbitrary abelian category, is a homological functor which sends pure triangles to short exact sequences, then there is a unique exact functor $H^{*}\colon\Mod{\T^{\c}}\to\A$ such that $H=H^{*}\circ \y$. 

In our setting, when $H\colon\T\to\scr{A}$ is coherent \cref{moduleextension} provides a unique definable functor $\bar{H}\colon\Mod{\T^{\c}}\to\scr{A}$ such that $\bar{H}\circ\y=H$. Any coherent functor preserves coproducts and sends pure triangles to short exact sequences. Consequently whenever $F$ is additionally homological the three functors
\[
\bar{H},H',H^{*}\colon\Mod{\T^{\c}}\to\scr{A}
\] 
given by \cref{moduleextension}, Krause's extension and Beligiannis's extension, respectively, all coincide, each by their own universal properties.

However, we do not require our original functor be homological, instead we just require it to be coherent. Beligiannis's condition of sending pure triangles to short exact sequences is, of course, a subrequirement of being coherent, as is Krause's requirement of preserving coproducts. It is interesting to note that combining these two conditions, and adding the preservation of products, enables one to obtain a lift without any requirement that the functor is homological.
\end{chunk}


\subsection{Transferring definability and endofinite objects}
We now show how, given a coherent functor $H\colon\T\to\scr{A}$, one can push forward and pull back information about definable subcategories.
\begin{prop}\label{coherentpreimageandimage}
Let $\scr{A}$ be a finitely accessible category with products, and $H\colon\T\to\scr{A}$ a coherent functor. 
\begin{enumerate}[label=(\arabic*)]
\item\label{coherentpreimage} If $\mc{D}\subseteq\scr{A}$ is a definable subcategory, then $H^{-1}\mc{D}$ is a definable subcategory of $\T$.
\item\label{coherentimage} If $\mc{D}\subseteq\T$ is a definable subcategory, then $\msf{pure}(H\mc{D})$, the closure of the image of $H\mc{D}$ under pure subobjects, is a definable subcategory of $\scr{A}$. Moreover, $\msf{pure}(H\mc{D})=\msf{pure}(\hat{H}\msf{Def}(\y\mc{D}))=\msf{Def}(H\mc{D})$.
\item\label{coherentdefbuild} For any class $\msf{X}\subseteq\T$, there is an inclusion $H(\msf{Def}(\msf{X}))\subseteq\msf{Def}(H\msf{X})$. Thus, if $\msf{X}$ definably builds $Z$, then $H\msf{X}$ definably builds $HZ$.
\end{enumerate}
\end{prop}
\begin{proof}
For \ref{coherentpreimage}, consider the definable functor $\hat{H}\colon \Flat{\T^\c} \to \scr{A}$ of \cref{uniprop}. As $\hat{H}$ is definable, $\hat{H}^{-1}(\mc{D})$ is a definable subcategory of $\msf{Flat}(\T^{\c})$ by~\cite[Proposition 13.3]{dac}. Therefore $\msf{y}^{-1}(\hat{H}^{-1}(\mc{D}))$ is a definable subcategory of $\T$ by \cref{purityflats}\ref{purityflats3}. Since $\hat{H}\y=H$, we have $H^{-1}(\mc{D}) = \y^{-1}(\hat{H}^{-1}(\mc{D}))$ which is definable as required. 

For \ref{coherentimage}, consider $\tilde{\mc{D}}:=\msf{Def}(\y\mc{D})$, the unique definable subcategory of $\Flat{\T^{\c}}$ corresponding to $\mc{D}$ under the bijection of \cref{purityflats}\ref{purityflats3}. Applying the definable lift $\hat{H}\colon\Flat{\T^{\c}}\to\scr{A}$ of \cref{uniprop}, gives that $\msf{pure}(\hat{H}\tilde{\mc{D}})$ is a definable subcategory of $\scr{A}$, by the paragraph after \cite[Corollary 13.4]{dac}. We now show that $\msf{Def}(H\mc{D})=\msf{pure}(\hat{H}\tilde{\mc{D}})=\msf{pure}(H\mc{D})$, which will conclude the proof.

To see this, first note that if $X\in \mc{D}$, then $H\mc{D}\subseteq \hat{H}\tilde{\mc{D}}$ as $\hat{H}\y X=HX$. Therefore $\msf{Def}(H\mc{D})\subseteq \msf{pure}(\hat{H}\tilde{\mc{D}})$, as the latter is definable. Now, suppose that $A\in\msf{pure}(\hat{H}\tilde{\mc{D}})$, so that there is a pure monomorphism $A \to \hat{H}Z$ for some $Z \in \tilde{\mc{D}}$. Now, by \cref{appendix.definable} together with \cref{purityflats}\ref{purityflats2}, there is a pure monomorphism $Z \to \prod_I \y E_i$ with each $E_i \in \mc{D} \cap \msf{pinj}(\T)$. Therefore as $\hat{H}$ is definable, we see that $A$ is a pure subobject of $\prod_I \hat{H}\y E_i \simeq H\prod_I E_i$, and hence $A \in \msf{pure}(H\mc{D})$ as required. Thus we have inclusions $\msf{Def}(H\mc{D})\subseteq\msf{pure}(\hat{H}\tilde{\mc{D}})\subseteq\msf{pure}(H\mc{D})$. Now, clearly $\msf{pure}(H\mc{D})\subseteq\msf{Def}(H\mc{D})$, and hence $\msf{Def}(H\mc{D})=\msf{pure}(\hat{H}\tilde{\mc{D}})=\msf{pure}(H\mc{D})$ as required. 

For \ref{coherentdefbuild}, since $Z \in \msf{Def}(\msf{X})$ we have $\y Z \in \msf{Def}(\y \msf{X})$ by \cref{purityflats}\ref{purityflats3}. Therefore $HZ = \hat{H}\y Z \in \hat{H}(\msf{Def}(\y\msf{X}))$. By \ref{coherentimage}, the latter is contained in $\msf{Def}(H\msf{X})$ which concludes the proof.
\end{proof}

Lastly, we show how endofiniteness is respected by coherent functors. Recall that, in a finitely accessible category with products $\scr{A}$, an object $E\in\scr{A}$ is \emph{endofinite} if $\t{length}_{\t{End}_{\scr{A}}(E)}\Hom_{\T}(A,E)<\infty$ for all $A\in\msf{fp}(\scr{A})$. If $\T$ is a compactly generated triangulated category, an object $X\in\T$ is \emph{endofinite} if $\t{length}_{\t{End}_{\T}(X)}\Hom_{\T}(A,X)<\infty$ for all $A\in\T^{\c}$.

\begin{prop}\label{cohpreservesendofin}
Let $H\colon\T\to\scr{A}$ be a coherent functor. If $X\in\T$ is endofinite, then $HX\in\scr{A}$ is endofinite.
\end{prop}

\begin{proof}
Let $\widehat{H}\colon\Flat{\T^{\c}}\to\scr{A}$ be the unique definable lift of $H$ as in \cref{uniprop}. Since $\hat{H}$ is a definable functor between finitely accessible categories, it preserves endofinite objects by \cite[Lemma 6.8]{BWZ}. As $\y$ establishes an equivalence of categories between endofinite objects in $\T$ and endofinite objects in $\Flat{\T^{\c}}$, see for example \cite[Lemma 10.6]{Beligiannis}, the result follows as $H = \hat{H} \circ \y$.  
\end{proof}

\subsection{Examples of coherent functors}\label{cohexamples}
In this section we give some examples of coherent functors.
\begin{chunk}\textbf{The restricted Yoneda embedding.}\label{uniyoneda}
Throughout we have used the fact that $\y\colon\T\to\Flat{\T^{\c}}$ transfers purity. More specifically, it preserves products and coproducts, and sends pure triangles to pure exact sequences since $\mrm{Ext}^1(f,\y X)=0$ for any finitely presented functor $f$ by \cref{flatequivalences}. Consequently it is a coherent functor. Moreover, by \cref{uniprop} it is the universal coherent functor for any finitely accessible category with products $\scr{A}$, in the sense that any coherent functor factors through it.
\end{chunk}

\begin{chunk}\textbf{Homological functors into Grothendieck categories.}\label{intoAB5}
Let $\A$ be a finitely accessible Grothendieck category, and suppose that $H\colon\T\to\A$ is a homological functor. The property of being homological renders one criterion for $H$ to be coherent redundant, namely that of asking $H$ to send pure triangles to pure exact sequences. 

Indeed, suppose that $H$ preserves coproducts, then by \cite[Corollary 2.4]{krsmash}, there is a unique colimit preserving functor $H'\colon\Mod{\T^{\c}}\to\A$ such that $H'\circ\y=H$. If $X\to Y\to Z$ is a pure triangle in $\T$, then $0\to \y X\to \y Y\to \y Z\to 0$ is pure exact in $\Mod{\T^{\c}}$ by \cref{definitionpureexact}. Since a short exact sequence is pure if and only if it is a direct limit of split exact sequences,  it follows that $0\to H'\y X\to H'\y Y\to H'\y Z\to 0$ is a pure exact sequence in $\A$ as $H'\colon \Mod{\T^\c} \to \A$ preserves colimits. Yet, by construction, this is the same as saying that $0\to HX\to HY\to HZ\to 0$ is pure exact; in other words $H$ sends pure triangles to pure exact sequences. Therefore, we deduce that a homological functor into a finitely accessible Grothendieck category $\A$ is coherent if and only if it preserves products and coproducts.
\end{chunk}

\begin{chunk}\textbf{$\msf{t}$-structures.}\label{tstructure}
A natural functor from a triangulated category to an abelian category is the homology associated to a t-structure. Suppose that the homological functor $\mrm{H}_{\mrm{t}}\colon \T \to \T^\heartsuit$ is coherent; not every t-structure has this property, but ones that do appear in nature, for example t-structures associated to classical silting sets as in \cite[Proposition 2]{NSZ}. We briefly demonstrate how the above results show that such t-structures satisfy desirable properties. 

Consider the induced definable lift $\bar{\mrm{H}}_{\mrm{t}}\colon \Mod{\T^\c} \to \T^\heartsuit$ from \cref{moduleextension}. Since $\bar{\mrm{H}}_{\mrm{t}}$ is exact, its kernel is a Serre subcategory which is moreover definable by \cref{coherentpreimageandimage}. Note that $\mrm{Im}(\bar{\mrm{H}}_{\mrm{t}}) = \T^\heartsuit$, so that there is a localisation sequence
\[
\begin{tikzcd}[column sep=5em]
\msf{Ker}(\bar{\mrm{H}}_{\msf{t}}) \arrow[r,hookrightarrow, shift left= 4pt] \arrow[r, leftarrow, shift right=3pt]& \Mod{\T^{\c}} \arrow[r, shift left= 4pt, "\bar{\mrm{H}}_{\msf{t}}"] \arrow[r, hookleftarrow, shift right=3pt,swap, "\rho"] & \T^\heartsuit.
\end{tikzcd}
\]
Since $\bar{\mrm{H}}_{\mrm{t}}$ has an exact left adjoint $\Lambda$ by \cref{moduleextensionleftadjoint}, this localisation sequence is part of a recollement, and thus $\Lambda$ is fully faithful. Since $\Lambda$ is exact, it also preserves equalisers. Therefore by the dual Beck's monadicity theorem~\cite[Lemma 4.1]{BalmerCameron}, we deduce that the adjunction $(\Lambda, \bar{\mrm{H}}_{\mrm{t}})$ is comonadic, which gives an algebraic presentation of the heart in terms of the module category:
\[\T^\heartsuit \simeq \msf{Comod}_{\Mod{\T^\c}}(\Lambda \circ \bar{\mrm{H}}_{\mrm{t}}).\]
\end{chunk}

\begin{chunk}\textbf{Coherent functors into $\ab$.}\label{abcoh}
The coherent functors into $\ab$ are precisely the coherent functors of \cref{coherent}. This is the contents of \cite[Theorem A and Proposition 5.1]{krcoh}, which shows that an additive functor $F\colon\T\to\ab$ is coherent, in the sense of \cref{coherent}, if and only if it preserves products and sends filtered homology colimits to colimits. As such, combining this with~\cref{enhanceddirectlimits}, one sees that $F$ is in $\msf{Coh}(\T,\ab)$ if and only if it is coherent in the sense of \cref{coherent}. In this case, one can completely characterise the homological coherent functors. They are precisely the functors of the form $\Hom_{\T}(C,-)$ where $C\in\T^{\c}$. This can be deduced immediately from Brown representability for the dual, and the definition of compact objects. This characterisation of homological coherent functors also appears as \cite[Proposition 2.9]{krsmash}.
\end{chunk}


\section{Definable functors}
In this section we make the key definition of the paper, that of a definable functor between compactly generated triangulated categories, before investigating properties of such functors. Let us motivate the coming definition. A definable functor should be a functor $\T\to\U$ which preserves the pure structure. As highlighted in \cref{purityflats}, the pure structures on $\T$ and $\Flat{\T^{\c}}$ are the same. Consequently, any functor $\T\to\U$ which preserves this structure would, ideally, relate to a purity preserving functor $\Flat{\T^{\c}}\to\Flat{\U^{\c}}$, which is just a definable functor in the sense of \cref{definablelfp}. We formalise this motivation in the following definition.

\begin{defn}\label{definitionofdefinable}
Let $F\colon\T\to\U$ be an additive functor between compactly generated triangulated categories. We say that $F$ is a \emph{definable functor} if there exists a definable functor $\hat{F}\colon \Flat{\T^{\c}}\to\Flat{\U^{\c}}$ such that $\msf{y} \circ F = \hat{F} \circ \msf{y}$. We then say that $\hat{F}$ is a \emph{definable lift} of $F$.
\end{defn}

\begin{rem}Let us firstly address some questions the reader may have about this definition.
\begin{enumerate}
\item The above definition may seem unsatisfactory, since it does not immediately provide an intrinsic characterisation of definable functors in terms of properties of the triangulated category. However, we will show later that the above definition is equivalent to $F$ satisfying some preservation properties which can be checked entirely on the triangulated level, see \cref{enhancement}.
\item When $F$ is definable, a definable lift $\hat{F}$ is in fact unique as we will show in \cref{uniquenesshat}.
\item The reasons for only asking for a lift to flat functors rather than the whole module category mirrors that in \cref{whyflatsremark}.

\end{enumerate}
\end{rem}

\begin{rem}
	One can use a similar argument to the proof of \cref{uniprop} to show that if $F,G\colon\T\to\U$ are definable functors and $\eta\colon F\Rightarrow G$ is a natural transformation, then there is a natural transformation $\hat{\eta}\colon\hat{F}\Rightarrow\hat{G}$ between their definable lifts. Consequently, the assignment sending a triple $(\T,F,\eta)$ to $(\Flat{\T^{\c}},\hat{F},\hat{\eta})$ gives a 2-functor from the 2-category of compactly generated triangulated categories and definable functors to the 2-category of finitely accessible categories with products and definable functors.
\end{rem}

In the following subsections we investigate the previous definition in more detail. In the first subsection, we study the implications on the functor $F$ itself, and then we investigate properties of the lift $\hat{F}$. Afterwards, we provide some criteria to check whether a given functor is definable, and then relate definable lifts to lifts of triangulated functors.


\subsection{Properties of definable functors}
Let us first give some elementary properties of definable functors, showing that they behave well with respect to purity and definable subcategories.

\begin{lem}\label{definablepreserves}
Let $F\colon\T\to\U$ be a definable functor. Then
\begin{enumerate}[label=(\arabic*)]
\item\label{coproducts} $F$ preserves coproducts;
\item\label{products} $F$ preserves products;
\item\label{pure} $F$ preserves pure triangles;
\item\label{phantoms} $F$ preserves phantom maps.
\item\label{endofinite} $F$ preserves endofinite objects.
\end{enumerate}
\end{lem}

\begin{proof}
For \ref{coproducts}, suppose $\{X_{i}\}_{I}$ is a set of objects in $\T$. There are natural isomorphisms
\begin{align*}
\msf{y}F(\oplus_{I} X_i) &\simeq \hat{F}\msf{y}(\oplus_I X_{i})  \\
&\simeq \oplus_{I} \hat{F}\msf{y}X_{i}, \t{ as $\msf{y}$ and $\hat{F}$ preserve coproducts} \\
&\simeq \oplus_{I} \msf{y}FX_{i} \\
&\simeq \msf{y}(\oplus_{I} FX_{i}), \t{ again, as $\msf{y}$ preserves coproducts.}
\end{align*}
Now, the above isomorphism is precisely the image of the canonical map $\oplus_{I}F(X_{i})\to F(\oplus X_{i})$ under $\msf{y}$; since $\msf{y}$ is conservative, it follows that $F$ preserves coproducts, which concludes the proof of \ref{coproducts}. The proof of \ref{products} is essentially identical to the proof of \ref{coproducts} so we omit it.

Let us now prove \ref{pure}. Let $X\xrightarrow{\alpha}Y\xrightarrow{\beta}Z$ be a pure triangle, which is equivalent to $0\to\y X\xrightarrow{\y\alpha}\y Y\xrightarrow{\y\beta}\y Z\to 0$ being a pure exact sequence in $\Flat{\T^{\c}}$. Since we have not assumed $F$ is triangulated, it is not a priori the case that $FX\to FY\to FZ$ is itself a triangle, but it is, by the functoriality of $F$, a candidate triangle. We show it is actually isomorphic to a pure triangle.

As $\hat{F}\colon\Flat{\T^{\c}}\to\Flat{\U^{\c}}$ is definable, the sequence $0\to \hat{F}\y X\xrightarrow{\hat{F}\y\alpha}\hat{F}\y Y\xrightarrow{\hat{F}\y\beta}\hat{F}\y Z\to 0$ is pure exact in $\Flat{\T^{\c}}$. In particular, since $\hat{F}\y\alpha = \y F\alpha$, we see that $F\alpha$ is a pure monomorphism in $\U$. Complete $F\alpha$ to a triangle $FX\xrightarrow{F\alpha}FY\xrightarrow{\gamma}C$, which is pure by construction. Then there is, by the (TR3) axiom, a diagram of candidate triangles
\[
\begin{tikzcd}
FX \arrow[d, equal] \arrow[r, "F\alpha"] & FY \arrow[d, equal] \arrow[r, "F\beta"] & FZ \arrow[d, dashed, "\delta"] \\
FX \arrow[r, "F\alpha"] & FY \arrow[r, "\gamma"] & C
\end{tikzcd}
\]
which commutes. Since the first two arrows are isomorphisms, so is the third. In particular, we see that the top row is a pure triangle, as desired. 

Part \ref{phantoms} follows from the fact that a map is phantom if and only if it is the connecting morphism in a pure triangle, together with \ref{pure}.

For \ref{endofinite}, we have that $\widehat{F}\y\colon\T\to\Flat{\U^{\c}}$ is coherent, so preserves endofinite objects by \cref{cohpreservesendofin}. As $\widehat{F}\y=\y F$, it follows that $\y F(X)$ is endofinite for any endofinite $X$, and thus, by \cite[Lemma 10.6]{Beligiannis}, $X$ is endofinite as claimed.
\end{proof}

\begin{rem}
    We note that \cref{definablepreserves}\ref{endofinite} is a non-triangulated generalisation of \cite[Lemma 3.8]{BIKP}.
\end{rem}

From the above, we immediately deduce the following. 
\begin{cor}\label{preservespurestructure}
A definable functor $F\colon\T\to\U$ preserves pure triangles and pure injective objects.
\end{cor}
\begin{proof}
Preservation of pure triangles is \cref{definablepreserves}\ref{pure}, while the preservation of pure injective objects follows from the Jensen-Lenzing criterion (see \cref{pureinjectiveequivalences}).
\end{proof}

In the subsequent section, we will show that conditions \ref{coproducts}, \ref{products}, and \ref{pure} of \cref{definablepreserves} actually characterise definable functors. Before that, let us show that definable functors behave as desired in relation to definable subcategories of $\T$ and $\U$.

\begin{prop}\label{preimageandimage}
Let $F\colon\T\to\U$ be a definable functor.
\begin{enumerate}[label=(\arabic*)]
\item\label{preimagedefinable} If $\mc{C}$ is a definable subcategory of $\U$, then the subcategory $F^{-1}\mc{C}=\{X\in\T:FX\in\mc{C}\}$ is a definable subcategory of $\T$.
\item\label{psoimage} If $\mc{D}$ is a definable subcategory of $\T$, then the closure of $F\mc{D}$ under pure subobjects in $\U$, denoted $\msf{pure}(F\mc{D})$, is a definable subcategory of $\U$.
\end{enumerate}
\end{prop}

\begin{proof}
For \ref{preimagedefinable}, consider the unique definable subcategory $\msf{Def}(\msf{y}\mc{C})$ of $\Flat{\U^{\c}}$ corresponding to $\mc{C}$ as in \cref{purityflats}\ref{purityflats3}. Then as $\hat{F}\y$ is coherent, by \cref{coherentpreimageandimage}\ref{coherentpreimage},  $\y^{-1}\hat{F}^{-1}(\msf{Def}(\msf{y}\mc{C}))$ is a definable subcategory of $\T$. Now we have
\[
\msf{y}^{-1}(\hat{F}^{-1}(\msf{Def}(\msf{y}\mc{C}))) = F^{-1}(\msf{y}^{-1}(\msf{Def}(\msf{y}\mc{C}))) = F^{-1}\mc{C}
\]
by using that $\hat{F}\msf{y} = \msf{y}F$, and \cref{purityflats}\ref{purityflats3} in turn. Therefore $F^{-1}\mc{C}$ is definable as required.

For \ref{psoimage}, set $\tilde{\mc{D}}=\msf{Def}(\y\mc{D})$. By considering the coherent functor $\y\circ F\colon\T\to\Flat{\U^{\c}}$, the subcategory $\msf{pure}(\hat{F}\tilde{\mc{D}})$ is definable by \cref{coherentpreimageandimage}\ref{coherentimage}. By \cref{purityflats}\ref{purityflats3}, it suffices to show that 
\[
\msf{pure}(F\mc{D})=\msf{y}^{-1}(\msf{pure}(\hat{F}\tilde{\mc{D}})).
\]
If $X\in\msf{pure}(F\mc{D})$, then there is an object $D\in\mc{D}$ and a pure monomorphism $X\to FD$. Then $\msf{y}X\to\msf{y}FD$ is a pure monomorphism by \cref{definitionpureexact}, and, as $\msf{y}FD\simeq \widehat{F}\msf{y}D$, we see that $\msf{y}X\in\msf{pure}(\widehat{F}\tilde{\mc{D}})$. For the other direction, if $\msf{y}X\in\msf{pure}(\widehat{F}\tilde{\mc{D}})$ then there is a pure monomorphism $\msf{y}X\to \widehat{F}\tilde{D}$, with $\tilde{D}\in \tilde{\mc{D}}$. Now, by \cref{appendix.definable} together with \cref{purityflats}\ref{purityflats2}, there is a pure monomorphism $\tilde{D}\to \prod_{i}\msf{y}E_{i}$ with each $E_{i}\in\mc{D}\cap\msf{pinj}(\T)$, and thus, as $\widehat{F}$ is definable, we see that $\msf{y}X$ is a pure subobject of $\prod_{I}\widehat{F}\msf{y}E_{i}$. Since $\prod_{I}\widehat{F}\msf{y}E_{i}\simeq \msf{y}F\prod_{I}E_{i}$, and as $F\prod_{I} E_{i}$ is pure injective by \cref{preservespurestructure}, there are isomorphisms
\[
\Hom_{\Mod{\U^{\c}}}(\msf{y}X,\prod_{I}\widehat{F}\msf{y}E_{i})\simeq \Hom_{\Mod{\U^\c}}(\msf{y}X, \msf{y}F\prod_{I} E_{i}) \simeq \Hom_{\U}(X,F\prod_{I}E_{i})
\] 
by \cref{definablepreserves}\ref{products} and \cref{pureinjectiveequivalences}. As such, there is a map $X\to F\prod_{I}E_{i}$ which is, by the above, a pure monomorphism. As $F\prod_{I}E_{i}\in F\mc{D}$, the claim is proved.
\end{proof}

As a special case of \cref{preimageandimage}\ref{preimagedefinable} we have the following.
\begin{cor}\label{kernelofdefinable}
Let $F:\T\to\U$ be a definable functor. Then $\msf{Ker}(F):=\{X\in\T:FX=0\}$ is a definable subcategory of $\T$. \qed
\end{cor}

Let us now show how definable functors interact with definable building in the sense of \cref{definablebuilding}.

\begin{lem}\label{deffunctorspreservebuilding}
Let $F\colon\T\to\U$ be a definable functor, and suppose that $\msf{X}\subset\T$. If $\msf{X}$ definably builds $T\in\T$, then $F\msf{X}=\{FX:X\in\msf{X}\}$ definably builds $FT$.
\end{lem}

\begin{proof}
By \cref{purityflats}\ref{purityflats3}, $\msf{X}$ definably builds $T$ if and only if $\y\msf{X}$ definably builds $\y T$ inside $\Flat{\T^{\c}}$; in other words, $\y T$ is obtained from closing $\y\msf{X}$ under products, filtered colimits and pure subobjects. Since $F$ is definable, there is a definable functor $\hat{F}\colon \Flat{\T^{\c}}\to\Flat{\U^{\c}}$ such that $\hat{F}\y\simeq\y F$. As $\hat{F}$ is definable, it preserves filtered colimits, products and pure subobjects. Consequently, $\hat{F}\y T$ is definably built by $\hat{F}\y\msf{X}$, which is equivalent to saying that $\y FT$ is definably built by $\y F\msf{X}$. Another application of \cref{purityflats}\ref{purityflats3} shows this is equivalent to $F\msf{X}$ definably building $FT$.
\end{proof}


\subsection{Uniqueness of the definable lift} 
We now discuss the properties of the lift $\hat{F}$ in the definition of a definable functor, showing in particular that it is unique, and also has some other convenient features. Firstly, we show that the definable lift $\hat{F}$ is unique.
\begin{prop}\label{uniquenesshat}
	If $F$ is definable, then there is a unique definable lift $\hat{F}$. That is, if $G\colon\Flat{\T^{\c}}\to\Flat{\U^{\c}}$ is definable such that $G \circ \msf{y}=\msf{y} \circ F$, then $G=\hat{F}$.
\end{prop}
\begin{proof}
The composite $\y \circ F\colon \T \to \Flat{\U^\c}$ is coherent by~\cref{definablepreserves}. Therefore the claim follows from \cref{uniprop}.
\end{proof}
Despite the fact that the definable lift $\hat{F}\colon \Flat{\T^{\c}}\to\Flat{\U^{\c}}$ only goes between flat functors, we may, as in the case of coherent functors, extend it uniquely to a definable functor between the entire module categories.

\begin{lem}\label{leftexactextension}
	If $F\colon \T\to\U$ is definable, then there is a unique left exact definable functor $\bar{F}\colon\Mod{\T^{\c}}\to\Mod{\U^{\c}}$ extending $\hat{F}$. 
\end{lem}
\begin{proof}
As $\msf{fpInj}(\Mod{\T^\c}) = \Flat{\T^\c}$ and similarly for $\U$ by \cref{flatequivalences}, there exists a unique left exact definable functor $\bar{F}\colon \Mod{\T^\c} \to \Mod{\U^\c}$ extending $F$ by \cite[Corollary 10.5]{Krauselfp}. 
\end{proof}

In fact, it is not just the case that $\bar{F}$ preserves kernels - it preserves all limits. By \cite[Corollary 10.5]{Krauselfp}, it has a left adjoint, which has pleasant properties.

\begin{prop}\label{barFhasleftadjoint}
	Let $F\colon \T \to \U$ be a definable functor. Then $\bar{F}\colon \Mod{\T^\c} \to \Mod{\U^\c}$ admits a left adjoint $\Lambda\colon \Mod{\U^\c} \to \Mod{\T^\c}$ which is exact and preserves finitely presented modules. \qed
\end{prop}

Consequently, given a definable functor $F\colon\T\to\U$, we obtain unique definable functors $\hat{F}\colon\Flat{\T^{\c}}\to\Flat{\U^{\c}}$ and $\bar{F}\colon\Mod{\T^{\c}}\to\Mod{\U^{\c}}$, such that $\hat{F}\y=\y F$ and $\bar{F}$ agrees with $\hat{F}$ on flat functors:
\[
\begin{tikzcd}[column sep=1in, row sep=1cm]
\T \arrow[r, "F"] \arrow[d,swap, "\y"] & \U \arrow[d, "\y"] \\
\Flat{\T^{\c}} \arrow[r, "\exists!\hat{F}"] \arrow[d,hookrightarrow]& \Flat{\U^{\c}} \arrow[d,hookrightarrow] \\
\Mod{\T^{\c}} \arrow[r, shift right = 1ex,swap, "\exists!\bar{F}"] \arrow[r, leftarrow, shift left=1ex, "\Lambda"] & \Mod{\U^{\c}}
\end{tikzcd}
\]

Since $\Lambda$ preserves finitely presented objects, if $\mc{S}$ is a Serre subcategory of $\mod{\T^{\c}}$, then the preimage $\Lambda^{-1}\mc{S}\cap\mod{\U^{\c}}$ is also a Serre subcategory, which therefore corresponds to a definable subcategory of $\U$. We now show how these definable subcategories obtained through preimages relate to those given by the images of definable subcategories of $\T$ under $F$, as detailed in \cref{preimageandimage}\ref{psoimage}.
\begin{cor}\label{imageofpureF}
	Let $F\colon \T \to \U$ be a definable functor, and $\mc{D}$ be a definable subcategory of $\T$. Then \[\msf{pure}(F\mc{D}) = \scr{D}(\Lambda^{-1}\scr{S}(\mc{D}) \cap \mod{\U^\c}).\]
\end{cor}
\begin{proof}
	Since $\Lambda$ preserves finitely presented modules by 
	\cref{barFhasleftadjoint}, $\Lambda^{-1}\scr{S}(\mc{D}) \cap \mod{\U^\c}$ is a Serre subcategory of $\mod{\U^\c}$. Therefore by the fundamental correspondence, the claim is equivalent to the statement that $\scr{S}(\msf{pure}(F\mc{D})) = \Lambda^{-1}\scr{S}(\mc{D}) \cap \mod{\U^\c}$. Let $f \in \mod{\U^\c}$. Then $f \in \Lambda^{-1}\scr{S}(\mc{D})$ if and only if $\Hom(\Lambda f, \y X) = 0$ for all $X \in \mc{D}$. By adjunction, this is equivalent to $\Hom(f, \y FX) = \Hom(f, \bar{F}\y X)=0$ for all $X \in D$. Since $\Hom(f,\y(-))\colon \U \to \ab$ is a coherent functor, this is moreover equivalent to $\Hom(f, \y Z) = 0$ for all $Z \in \msf{pure}(F\mc{D})$ as required.
\end{proof}

We have already seen in \cref{barFhasleftadjoint} that $\bar{F}$ always has a left adjoint, and we end this subsection by giving a condition under which it also has a right adjoint.

\begin{prop}\label{barFhasrightadjoint}
	Let $F\colon \T \to \U$ be a definable functor, and suppose that $F$ has a left adjoint $L$. Then there is an adjoint triple $\Lambda \dashv \bar{F} \dashv \rho$. In particular, $\bar{F}$ is exact.
\end{prop}
\begin{proof}
	Since $F$ preserves coproducts by \cref{definablepreserves}, its left adjoint $L$ preserves compact objects. As such, we may consider the functor $- \circ L\vert_{\U^\c}\colon \Mod{\T^\c} \to \Mod{\U^\c}$, and note that this has both a left and right adjoint. Therefore, $- \circ L\vert_{\U^\c}$ is a definable functor, and moreover commutes with the restricted Yoneda embeddings. As such, we have $\bar{F} = - \circ L\vert_{\U^\c}$ by the uniqueness of $\bar{F}$ as in \cref{leftexactextension}. In particular, $\bar{F}$ has both a left and a right adjoint given by the left and right Kan extension. (Note that the left adjoint may be identified with $\Lambda$ in the sense of \cref{barFhasleftadjoint} by uniqueness of adjoints.)
\end{proof}


\subsection{Criteria for definability}
In this section, we give various criteria for determining whether or not a given functor is definable. These are intrinsic characterisations, which can be checked on the triangulated level. Firstly, we require the following lemma.
\begin{lem}\label{enhancedpreservation}
Let $F\colon \T \to \U$ be a functor between compactly generated triangulated categories. Then the following are equivalent:
\begin{enumerate}[label=(\arabic*)]
\item\label{hocolim} $F$ preserves products and filtered homology colimits;
\item\label{coprodandpure} $F$ preserves coproducts, products, and pure triangles.
\end{enumerate}
\end{lem}
\begin{proof}
Apply \cref{enhanceddirectlimits} to $\msf{y} \circ F$ and use conservativity of the restricted Yoneda embedding.
\end{proof}

Our strongest criterion for checking definability is the following, which gives some equivalent characterisations of definable functors.
\begin{thm}\label{enhancement}
Let $F\colon \T \to \U$ be a functor between compactly generated triangulated categories. Then the following are equivalent:
\begin{enumerate}[label=(\arabic*)]
\item\label{TFAE1} $F$ is definable;
\item\label{TFAE2} $F$ preserves products and filtered homology colimits;
\item\label{TFAE3} $F$ preserves coproducts, products, and pure triangles.
\end{enumerate}
Moreover, if $F$ arises from a functor of stable $\infty$-categories then these are also equivalent to:
\begin{enumerate}[label=(\arabic*)]
	\setcounter{enumi}{3}
\item\label{TFAE4} $F$ preserves products and filtered homotopy colimits.
\end{enumerate}
\end{thm}
\begin{proof}
That \ref{TFAE1} implies \ref{TFAE3} is the content of \cref{definablepreserves}, and \ref{TFAE2} and \ref{TFAE3} are equivalent by \cref{enhancedpreservation}. We now show that \ref{TFAE3} implies \ref{TFAE1}. By the assumptions on $F$, the composite $\msf{y}\circ F\colon \T \to \Flat{\U^\c}$ is $\Flat{\U^\c}$-coherent. Therefore by \cref{uniprop} there exists a definable functor $\hat{F}\colon \Flat{\T^\c} \to \Flat{\U^\c}$ such that $\hat{F} \circ \msf{y} = \msf{y} \circ F$ as required. Finally, assume that $F$ arises from a functor of stable $\infty$-categories. Then \ref{TFAE3} and \ref{TFAE4} are equivalent by a straightforward modification of the argument given in \cref{enhancedpreservation}.
\end{proof}

In the definition of definable functor, we required the lift to preserve flat objects. We note that given a definable functor on the whole module categories which lifts $F$, this is immediate.
\begin{lem}\label{preservesflats}
Let $F\colon \T \to \U$ be a functor between compactly generated triangulated categories. If there exists a definable functor $\bar{F}\colon \Mod{\T^\c} \to \Mod{\U^\c}$ such that $\bar{F} \circ \msf{y} = \msf{y} \circ F$, then $\bar{F}$ preserves flats, and $F$ is definable.
\end{lem}
\begin{proof}
As products and direct limits in $\Flat{\T^\c}$ are computed in $\Mod{\T^\c}$, it suffices to check that $\bar{F}$ preserves flats. If $X\in\Flat{\T^{\c}}$, then $X\simeq\rlim_{I}\msf{y}A_{i}$, where each $A_{i}\in\T^{\c}$ by \cref{flatequivalences}. Then $\bar{F}(X)\simeq\rlim\bar{F}\msf{y}A_{i}\simeq\rlim\msf{y}FA_{i}$. Although each $FA_{i}$ need not be a compact in $\U$, each $\msf{y}FA_{i}$ is an object in $\Flat{\U^{\c}}$, which is closed under direct limits, hence $\bar{F}X\in\Flat{\U^{\c}}$ as required.
\end{proof}


\subsection{Triangulated functors}
In this section, we comment on the particularly pleasant case of triangulated definable functors. In this setting we may apply adjoint functor theorems from Brown representability to obtain simple characterisations of triangulated definable functors, together with various additional consequences. This means that many of the results in this subsection follow easily, but we include them for completeness.

Recall from the introduction, that a triangulated functor between compactly generated triangulated categories is \emph{Beligiannis definable} if and only if it preserves coproducts and products, see~\cite[Definition 6.11]{Beligiannis}. 

\begin{prop}\label{triangulatedversion}
Let $F\colon\T\to\U$ be a triangulated functor between compactly generated triangulated categories. Then $F$ is definable if and only if $F$ preserves products and coproducts. In other words, $F$ is definable if and only if it is Beligiannis definable. Moreover in this setting, the definable lift $\bar{F}\colon \Mod{\T^\c} \to \Mod{\U^\c}$ is exact.
\end{prop}

\begin{proof}
The forward implication is contained in \cref{definablepreserves}. For the converse, applying the universal property of $\y$, discussed in \cref{universalpropertyofyoneda}, to the composite $\msf{y}_\U \circ F\colon \T \to \Mod{\U^\c}$ we obtain an exact, coproduct preserving functor $\bar{F}\colon\Mod{\T^\c} \to \Mod{\U^\c}$ so that $\msf{y}_\U \circ F = \bar{F} \circ \msf{y}_\T$. We next show that it is a definable functor.

Since $F$ is product preserving and triangulated, it has a triangulated left adjoint $L\colon \U \to \T$ by \cref{triangulatedbackground}. As $F$ preserves coproducts, $L$ preserves compacts by \cref{triangulatedbackground}, and as such there is an exact, colimit preserving functor $-\circ L\vert_{\U^\c}\colon \Mod{\T^\c} \to \Mod{\U^\c}$, where $L\vert_{\U^\c}$ is the restriction of $L$ to $\U^\c$. Applying the uniqueness part of the universal property in \cref{universalpropertyofyoneda}, we see that $\bar{F} = - \circ L\vert_{\U^\c}$. Therefore $\bar{F}$ has a left adjoint given by the left Kan extension $(L\vert_{\U^\c})_!$, and as such $\bar{F}$ preserves products and is definable. Applying \cref{preservesflats} shows that $F$ is therefore definable.
\end{proof}

\begin{rem}
As an immediate corollary of \cref{preservespurestructure} and \cref{triangulatedversion}, one sees that a coproduct and product preserving triangulated functor also preserves pure triangles. This can also be seen more directly, by checking that $F$ preserves phantom maps, which can be done by using \cref{triangulatedbackground} to see that $F$ has a compact preserving left adjoint.
\end{rem}

\begin{rem}
The main application of Beligiannis definable functors was to obtain embeddings of the Ziegler spectra~\cite[Theorem 6.13]{Beligiannis}. We consider maps on the Ziegler spectrum in \cref{sec:mapsonziegler}, enabling us to generalise this result.
\end{rem}

We will now explain the relationship between definable lifts and the universal property of the restricted Yoneda embedding as recalled in \cref{universalpropertyofyoneda}. This may help the reader orient themselves in comparing the universal property of the restricted Yoneda embedding proved in \cref{sec:universalprop} with that of \cref{universalpropertyofyoneda}. 

\begin{thm}\label{exdefredux}
Let $F\colon \T \to \U$ be a triangulated, definable functor. Write $L$ (resp., $R$) for the left (resp., right) adjoints of $F$, which exist by \cref{triangulatedbackground}, as $F$ preserves coproducts and products (see \cref{triangulatedversion}).
\begin{enumerate}[label=(\arabic*)]
\item\label{triitem1} There exists an adjoint triple
\[\begin{tikzcd}
\Mod{\T^\c} \ar[rr, "\bar{F}" description] & & \Mod{\U^\c} \ar[ll, yshift=2mm, "\Lambda"'] \ar[ll, yshift=-2mm, "\rho"]
\end{tikzcd}\]
such that:
\begin{enumerate}[label=(\alph*)]
\item\label{triitema} $\bar{F}$ is an exact, definable functor which preserves flats, and satisfies $\msf{y}_\U\circ\bar{F} = \bar{F}\circ\msf{y}_\T$;
\item\label{triitemb} $\Lambda$ is exact, preserves finitely presented objects and satisfies $\msf{y}_\T \circ L = \Lambda \circ \msf{y}_\U$.
\end{enumerate}
\item\label{triitem2} Suppose moreover that $F$ preserves compacts. Then $\rho = \bar{R}$ and is exact, definable, preserves flats and satisfies $\msf{y}_\T \circ \rho = R \circ \msf{y}_\U$. Moreover, $\rho = \bar{R}$ has a right adjoint $\beta$. 
\item\label{triitem3} The functor $\Lambda$ has a left adjoint if and only if $L$ preserves products. In this case, we have $\Lambda = \bar{L}$.
\end{enumerate}
\end{thm}
\begin{proof}
We first prove \cref{triitem1}. \cref{triitema} is contained in \cref{triangulatedversion}, and \cref{triitemb} follows from \cref{barFhasleftadjoint} together with the fact that it may be identified with the left Kan extension $(L\vert_{\U^\c})_!$ and hence commutes with the restricted Yoneda embeddings. Note that $\rho$ exists, as $\bar{F} = - \circ L\vert_{\U^\c}$, and hence has right adjoint $\rho$ given by the right Kan extension $(L\vert_{\U^\c})_*$.

We now prove \cref{triitem2}. Since $F$ is triangulated, definable, and compact-preserving, it has a right adjoint $R$ which is triangulated and definable by \cref{triangulatedbackground}. As such, applying \cref{triitem1} to $R$ we obtain an adjoint triple $\alpha \dashv \bar{R} \dashv \beta$. Therefore to prove \cref{triitem2}, it suffices to show that $\alpha = \bar{F}$. As both $\alpha$ and $\bar{F}$ are exact and preserve direct limits, it suffices to show that they agree on $\msf{y}T$ where $T\in \T^\c$. Then for any $U \in \U$, we have \[\Hom(\bar{F}\msf{y}T, \msf{y}U) = \Hom(\msf{y}FT, \msf{y}U) = \Hom(T, RU) = \Hom(\msf{y}T, \msf{y}RU) = \Hom(\msf{y}T, \bar{R}\msf{y}U) = \Hom(\alpha\msf{y}T, \msf{y}U)\] by adjunction and fully faithfulness of restricted Yoneda when the first variable is compact. Since $\alpha$ has a direct limit preserving right adjoint and $F$ preserves compacts, both $\bar{F}(\msf{y}T)$ and $\alpha(\msf{y}T)$ are finitely presented, and hence they are equal by~\cite[Theorem 7.2]{krcoh}, also see \cref{coherent}, as required.

For part \cref{triitem3}, if $L$ preserves products then $L$ is definable and compact-preserving so \cref{triitem2} shows that we obtain adjoints $\gamma \dashv \bar{L} \dashv \delta \dashv \eta$, and $\delta = \bar{F}$. Hence $\bar{L} = \Lambda$ by uniqueness of adjoints, so we see that $\Lambda$ has a left adjoint. Conversely, if $\Lambda$ preserves products, then \[\msf{y}L \prod X_i = \Lambda\msf{y}\prod X_i = \Lambda \prod \msf{y}X_i = \prod \Lambda \msf{y}X_i = \msf{y}\prod L X_i.\] One identifies this as the image under the restricted Yoneda embedding of the canonical map $L\prod X_i \to \prod LX_i$, and hence by the conservativity of $\y$ we see that $L$ preserves products.
\end{proof}


\section{Examples of definable functors}\label{sec:examples}

In this section we show how ubiquitous definable functors are. 

\begin{chunk}{\textbf{Deriving exact definable functors.}}
Let $F\colon\mathscr{A}\to\mathscr{B}$ be an exact functor of Grothendieck abelian categories that is definable. Then the induced functor $F\colon\D(\mathscr{A})\to \D(\mathscr{B})$ is a definable functor by \cref{triangulatedversion} as it is triangulated, and as products and coproducts are computed termwise. For instance, if $f\colon R\to S$ is a ring map, then the restriction of scalars functor $f^{*}\colon\D(S)\to\D(R)$ is definable.
\end{chunk}

\begin{chunk}{\textbf{Derived Hom and tensor.}}\label{tensorwithbounded}
Let $R$ be a ring. Then the functor 
\[
\msf{R}\Hom_{R}(C,-)\colon \D(R)\to\D(\Z)
\]
is definable if and only if $C$ is a perfect complex. Indeed, $\msf{R}\Hom_{R}(C,-)$ is triangulated and always preserves products, hence, by \cref{triangulatedversion}, we require it to preserve coproducts. This is equivalent, as $R$ generates $\D(R)$, to $C$ being compact. 

For the derived tensor product, if $X\in\D(R)$, then
\[
X\otimes_{R}^{\msf{L}}-\colon \D(R^{\op})\to\D(\Z)
\] 
is definable if and only if $X$ is compact. By definition, $X \otimes_R^\msf{L} - = \mc{F}_{X}\otimes_R -$ where $\mc{F}_{X}$ is a semi-flat replacement for $X$. Now $\mc{F}_X \otimes_R -\colon \C(R^{\op})\to\C(\ab)$ preserves products if and only if $\mc{F}_{X}$ is a bounded complex of finitely presented modules by \cite[Theorem C.9]{dcmca}. As semi-flat complexes are termwise flat, this is moreover equivalent to $X$ being compact.
\end{chunk}

\begin{chunk}{\textbf{Geometric functors in tensor-triangular geometry}.}\label{exgeometricmap}
Let $\T$ and $\U$ be rigidly-compactly generated tensor-triangulated categories. Recall that a functor $f^{*}\colon\T\to\U$ is called a \emph{geometric functor} if it is triangulated, preserves coproducts, and is strong monoidal. By \cref{triangulatedbackground}, any geometric functor $f^{*}$ admits a right adjoint
\[
f_{*}\colon\U\to\T.
\]
As $f^*$ is strong monoidal, it preserves compact objects since these coincide with the rigid objects, and hence the right adjoint $f_*$ also preserves coproducts. Therefore the functor $f_*$ is definable by \cref{triangulatedversion}; that is, the right adjoint to any geometric functor is definable.
This encompasses several examples of interest, for instance, the coinduction functor from $H$-spectra to $G$-spectra for $H$ a closed subgroup of a compact Lie group $G$, and the derived pushforward along a map of quasi-compact, quasi-separated schemes; see~\cite{BDS} for more examples and details. We consider the question of when $f^*$ is definable in the next example.
\end{chunk}

\begin{chunk}{\textbf{Grothendieck-Neeman duality.}}\label{GNduality}
Let $f^*\colon \T \to \U$ be a geometric functor between rigidly-compactly generated tensor-triangulated categories as in \cref{exgeometricmap}. Write $f^{(1)}$ for the right adjoint to $f_*$ (which exists since $f_*$ preserves coproducts). Then the following are equivalent:
\begin{enumerate}
	\item $f^*\colon \T \to \U$ is a definable functor; 
	\item $f^{(1)}\colon \T \to \U$ is a definable functor;
	\item Grothendieck-Neeman duality holds for $f^*$, i.e., there is a natural isomorphism $\omega_f \otimes f^*(-) \simeq f^{(1)}(-)$ where $\omega_f = f^{(1)}(\1_\U)$ is the relative dualising object. 
\end{enumerate}
The equivalence of these three conditions is an immediate consequence of \cref{triangulatedversion} together with~\cite[Theorem 3.3 and Remark 3.14]{BDS}. We give a couple of examples and refer the reader to~\cite{BDS} for more details and more examples. Given a map $f\colon R \to S$ of commutative rings, or (highly structured) commutative ring spectra, the extension of scalars $S \otimes_R^\msf{L} -\colon \msf{D}(R) \to \msf{D}(S)$ is definable if and only if $\msf{R}\mrm{Hom}_R(S,-)\colon \msf{D}(R) \to \msf{D}(S)$ is definable, if and only if $S$ is compact in $\msf{D}(R)$. If $f\colon X \to Y$ is a map of quasi-compact, quasi-separated schemes, then the inverse image functor $f^*\colon \msf{D}(\mrm{QCoh}(Y)) \to \msf{D}(\mrm{QCoh}(X))$ is definable if and only if the twisted inverse image functor $f^{(1)} = f^!$ is definable. These are moreover equivalent to $f$ being quasi-perfect in the sense of~\cite{LN} and to Grothendieck duality as in~\cite[Proposition 2.1]{LN}, also see~\cite{Neemanduality}.
\end{chunk}

\begin{chunk}{\textbf{Functors between Frobenius categories.}}\label{stabilize}
Let $\A$ be an exact category in the sense of Quillen. Recall that an object $E\in\A$ is said to be injective if every inflation $E\to X$ in $\A$ is split, and dually an object $P\in\A$ is said to be projective if every deflation $X\to P$ in $\A$ is split. The category $\A$ is said to have enough injectives if for every $A\in\A$ there is an inflation $A\to E$ with $E$ an injective object. The notion of enough projectives is dually defined.

The category $\A$ is said to be \ti{Frobenius} if it has enough injectives and enough projectives, and the classes of injective and projective objects coincide. In this case the stable category of $\A$, denoted $\msf{St}\,\A$, is triangulated. This category has objects the same as $\A$, and morphisms are given by
\[
\underline{\Hom}_{\A}(X,Y)=\Hom_{\A}(X,Y)/\{f\colon X\to Y: f\t{ factors through an injective object}\}.
\]
The shift of an object in $\msf{St}\,\A$ is given by its cosyzygy, that is $\Sigma X$ appears in a conflation $0\to X\to E\to \Sigma X\to 0$, where $E$ is an injective object. A sequence $X\to Y\to Z\to \Sigma X$ is a triangle in $\msf{St}\,\A$ if and only if $0\to X\to Y\to Z\to 0$ is a conflation in $\A$. Details and proofs of these facts can be found in \cite[\S 3.3]{krbook}.

Let us note that we shall always assume $\msf{St}\,\A$ is compactly generated, even though this, in general, need not be the case.

Now, suppose that $\A$ and $\B$ are exact categories, and let $F\colon\A\to\B$ be an additive functor between Frobenius categories. Suppose $F$ additionally sends injective objects in $\A$ to injective objects in $\B$. It then induces an additive functor
\[
\msf{St}\,F\colon\msf{St}\,\A\to\msf{St}\,\B,
\]
which need not be a triangulated functor, since we have made no assumptions on $F$ being exact.

Let $\A$ and $\B$ be Frobenius categories which are closed under direct limits and products, and suppose $F\colon\A\to\B$ is a functor which preserves injective objects, direct limits and products. We claim that the induced functor
\[
\msf{St}\,F\colon\msf{St}\,\A\to\msf{St}\,\B
\]
is a definable functor of triangulated categories which is not necessarily triangulated. To prove this, let $\msf{q}_\A\colon \A \to \msf{St}\,\A$ denote the canonical localisation functor, and likewise define $\msf{q}_{\B}$, and note that $\msf{St}\,F\circ\msf{q}_\A=\msf{q}_{\B}\circ F$. That $\msf{St}\,F$ preserves products follows since both $\msf{q}_\A$ and $\msf{q}_\B$, as well as $F$, preserve products, while the preservation of homotopy colimits follows from \cite[Proposition 2.2(2)]{toen} and the fact that $F$ preserves direct limits. Consequently, by \cref{enhancement}, we see that $\msf{St}\,F$ is, as desired, definable.
\end{chunk}

\begin{chunk}{\textbf{Extension for Gorenstein flat modules.}}\label{gfc}
This is a more detailed example concerning Frobenius categories, where the induced functor on stable categories is definable, but the functor on Frobenius categories is not. Let $R$ be a right coherent ring. Recall that a left $R$-module $M$ is \ti{Gorenstein flat} if there is an acyclic complex of flat left $R$-modules $\msf{F}$ with $M=Z_{0}\msf{F}$ such that $E\otimes_{R}\msf{F}$ is acyclic for every injective right $R$-module $E$. The full subcategory of Gorenstein flat (left) $R$-modules is denoted $\msf{GFlat}(R)$. 
A left $R$-module $X$ is \ti{cotorsion} if $\t{Ext}_{R}^{1}(F,X)=0$ for all flat left $R$-modules $F$. The full subcategory of cotorsion modules is denoted $\msf{Cot}(R)$. 

There is an abelian model structure, called the Gorenstein flat model structure, on $\Mod{R}$, as introduced in \cite[Theorem 3.3]{gillespie}. The class of bifibrant objects in this model structure is the class $\msf{GFlatCot}(R):=\msf{GFlat}(R)\cap\msf{Cot}(R)$, which, by \cite[Corollary 3.4]{gillespie}, is a Frobenius exact category, whose projective-injective objects is the category $\msf{FlatCot}(R)$, the category of flat and cotorsion $R$-modules, which are the trivial bifibrant objects. The homotopy category of this model structure is triangulated equivalent to the stable category $\msf{St}(\msf{GFlatCot}(R))$. As such, $\msf{St}(\msf{GFlatCot}(R))$ admits both products and coproducts. The products are those of $\Mod{R}$, while the coproducts arise through a functorial fibrant replacement.

Suppose that $R\to S$ is a ring map of noetherian rings such that both $R$ and $S$ admit dualising complexes, and $S$ is finitely generated of finite flat dimension over $R$. In this set up, the categories $\msf{St}(\msf{GFlatCot}(R))$ and $\msf{St}(\msf{GFlatCot}(S))$ are compactly generated. The extension of scalars functor $S\otimes_{R}-$ yields a functor $\msf{St}(\msf{GFlatCot}(R))\to\msf{St}(\msf{GFlatCot}(S))$ which is triangulated and definable, as was proved in~\cite{birdgfc}.

Let us now contrast this example with \cref{stabilize}. In this example, the Frobenius categories $\msf{GFlatCot}(R)$ and $\msf{GFlatCot}(S)$ are not themselves definable subcategories of $\Mod{R}$ and $\Mod{S}$; although in many desirable circumstances the functor $S\otimes_{R}-\colon\msf{GFlat}(R)\to\msf{GFlat}(S)$ is a definable functor between definable subcategories. The issue is that coproducts (and direct limits) in $\msf{GFlatCot}(R)$ are not those inherited from the ambient module category - a coproduct of cotorsion modules is seldom cotorsion. As such this example does not fall into the setting of \cref{stabilize}, and instead gives an `exotic' definable functor on stable categories which does not arise from a definable functor on the Frobenius level. Several questions pertaining to definability in $\msf{St}(\msf{GFlatCot}(R))$ and the behaviour of the functor $S\otimes_{R}-$ were considered in detail in~\cite{birdgfc}, particularly in relation to the Ziegler spectrum of the category of Gorenstein flat modules.
\end{chunk}

\begin{chunk}{\textbf{Definable truncations.}}
Let us now give an explicit example of a non-triangulated definable functor. Let $\msf{t}=(\T_{\geq 0},\T_{\leq 0})$ be a t-structure such that the coaisle $\T_{\leq 0}$ is a definable subcategory. Sources of t-structures with definable coaisles are discussed in~\cite[\S 8.2]{saorstov}. In this setup the composition $\msf{inc}\circ \tau_{\leq 0}\colon\T\to\T$ is a definable functor, but is not triangulated unless $\T_{\leq 0}$ is. To see this, it suffices to show that $\msf{inc} \circ \tau_{\leq 0}$ preserves products and filtered homology colimits by \cref{enhancement}. Product preservation is immediate, so we show closure under filtered homology colimits. Consider the triangles \[\msf{homcolim}\tau_{>0}X_i \to \msf{homcolim} X_i \to \msf{homcolim}\tau_{\leq 0} X_i\] and \[\tau_{>0}\msf{homcolim}X_i \to \msf{homcolim}X_i \to \tau_{\leq 0}\msf{homcolim}X_i.\] As the coaisle is definable by assumption, the latter term of the first triangle lies in it, since definable subcategories are closed under filtered homology colimits. The aisle is always closed under filtered homology colimits, so comparing the above triangles gives the desired claim. One can give similar examples based on the aisle rather than coaisle.
\end{chunk}

\begin{chunk}{\textbf{Embeddings along coherent functors.}}\label{stalk}
Using coherent functors, we may construct many non-triangulated definable functors. Suppose $H\colon\T\to\scr{A}$ is a coherent functor into a finitely accessible category with products. Recall from \cref{appendix.D} that there is functor $\msf{d}_{\scr{A}}\colon\scr{A}\to\mbb{D}(\scr{A})$ which preserves direct limits and products, so is a definable functor from $\scr{A}$ into a locally coherent category; thus the composition $\tilde{H}=\msf{d}_{\scr{A}}\circ H\colon\T\to\mbb{D}(\scr{A})$ is coherent.

Let us suppose that $\mbb{D}(\scr{A})$ is such that $\D(\scr{A}):=\D(\mbb{D}(\scr{A}))$, the derived category of $\mbb{D}(\scr{A})$, is compactly generated; conditions for such an assumption to hold can be found in \cite[\S 7]{spure}. Then the $n$-th stalk functor $\tilde{H}(-)[n]$ defined as the composite
\[\T \xrightarrow{\tilde{H}} \mbb{D}(\scr{A}) \xrightarrow{[n]} \D(\scr{A})\]
is a definable functor, which is not triangulated. 

Some of these steps may, of course, be superfluous, for example when $\scr{A}$ is itself a Grothendieck category such that $\D(\scr{A})$ is compactly generated. Then there is a coherent functor given by the homological functor $H\colon\D(\scr{A})\to\scr{A}$ associated to the standard $t$-structure. Then $H(-)[n]\colon\D(\scr{A})\to\D(\scr{A})$ is definable but not triangulated.
\end{chunk}

\section{Induced maps between Ziegler spectra}\label{sec:mapsonziegler} 
If $F\colon\T\to\U$ is a definable functor, then we saw in \cref{preservespurestructure} that $F$ preserves pure injective objects. One may then hope that, in certain circumstances, $F$ induces a map $\msf{Zg}(\T)\to\msf{Zg}(\U)$ between the Ziegler spectra of $\T$ and $\U$. In this section we prove that this is indeed the case. We first treat the case of coherent functors.
\begin{thm}\label{thm:coherentzieglermap}
Suppose $H\colon\T\to\scr{A}$ is a coherent functor which is full on pure injective objects. Then:
\begin{enumerate}
    \item if $X \in \msf{pinj}(\T)$ then $HX \in \msf{pinj}(\scr{A})$ or $HX = 0$; 
    \item $H$ induces a closed and continuous map $\msf{Zg}(\T)\to\msf{Zg}(\scr{A})$, which restricts to a homeomorphism 
\[
\msf{Zg}(\T)\setminus\msf{K}\xrightarrow{\sim}\msf{Zg}(\scr{A}) \cap \msf{Def}(\msf{Im}\,H).
\]
where $\msf{K} = \{X \in \msf{Zg}(\T) : HX = 0\}$.
    \end{enumerate}
\end{thm}

\begin{proof}
As $H$ is full on pure injectives, so is the induced definable lift $\widehat{H}\colon\Flat{\T^{\c}}\to\scr{A}$ of \cref{uniprop}, since for any pure injectives $E_1, E_2 \in \scr{A}$, we have \[\Hom_{\scr{A}}(\widehat{H}(E_{1}),\widehat{H}(E_{2}))\simeq \Hom_{\scr{A}}(HX_{1},HX_{2})\] where $X_{i}$ is the unique pure injective object in $\T$ such that $\y X_{i}=E_{i}$. For $X\in\msf{pinj}(\T)$, we deduce that $HX = \widehat{H}(\y X)$ is either zero or an indecomposable object of $\scr{A}$ by \cite[Corollary 13.6]{dac} as $\widehat{H}$ is full on pure injectives, which proves the first statement.

For the second statement, as $\widehat{H}$ is full on pure injectives, by \cite[Corollary 15.4]{dac}, it induces a closed and continuous map $\hat{H}\colon\msf{Zg}(\Flat{\T^{\c}})\to\msf{Zg}(\scr{A})$. Since $\y\colon\msf{Zg}(\T)\xrightarrow{\sim}\msf{Zg}(\Flat{\T^{\c}})$ is a homeomorphism by \cref{purityflats}\ref{purityflats4}, the fact that $H=\hat{H}\circ\y$ shows that $H\colon \msf{Zg}(\T) \to \msf{Zg}(\scr{A})$ is also closed and continuous.

Moreover the map $\widehat{H}\colon\msf{Zg}(\Flat{\T^{\c}})\to\msf{Zg}(\scr{A})$ gives, by \cite[Theorem 15.5]{dac}, a homeomorphism \[\msf{Zg}(\Flat{\T^{\c}})\setminus\{Z\in\msf{Zg}(\Flat{\T^{\c}}):\widehat{H}(Z)=0\}\simeq \msf{Zg}(\scr{A}) \cap \msf{Def}(\msf{Im}\,\widehat{H}).\] By \cref{purityflats}, we see that $\{Z\in\msf{Zg}(\Flat{\T^{\c}}):\widehat{H}(Z)=0\}\simeq \{X\in\msf{Zg}(\T):H(X)=0\}$. Moreover by \cref{coherentpreimageandimage}\ref{coherentimage}, it is immediate that $\msf{Def}(\msf{Im}\,H)=\msf{Def}(\msf{Im}\,\widehat{H})$, thus $H$ restricts to a homeomorphism $\msf{Zg}(\T)\setminus\msf{K}\simeq \msf{Zg}(\scr{A}) \cap \msf{Def}(\msf{Im}\,H)$ as claimed. 
\end{proof}

\begin{rem}\label{allresultstransfertocoherent}
We note that every result of \cite[\S15]{dac} can be transferred to the setting of coherent functors, by considering the definable functor $\widehat{H}\colon\Flat{\T^{\c}}\to\scr{A}$ of \cref{uniprop}, and the equivalence of categories $\y\colon\msf{Pinj}(\T)\xrightarrow{\sim}\msf{Pinj}(\Flat{\T^{\c}})$.
\end{rem}

We see, as in the case of coherent functors, that under a certain fullness condition definable functors also induce maps between Ziegler spectra.

\begin{thm}\label{thm:definducedonzeigler}
Let $F\colon\T\to\U$ be a definable functor which is full on pure injective objects. Then:
\begin{enumerate}
    \item if $X\in\msf{pinj}(\T)$, then $FX\in\msf{pinj}(\U)$ or $FX = 0$;
    \item $F$ induces a closed and continuous map $\msf{Zg}(\T)\to\msf{Zg}(\U)$ which restricts to a homeomorphism
\[
\msf{Zg}(\T)\setminus\msf{K}\simeq \msf{Zg}(\U) \cap \msf{Def}(\msf{Im}\,F)
\]
where $\msf{K} = \{X \in \msf{Zg}(\T) : FX = 0\}$.
\end{enumerate}
\end{thm}

\begin{proof}
We apply \cref{thm:coherentzieglermap} to the coherent functor $\widehat{F}\circ \y\colon\T\to\Flat{\U^{\c}}$, which proves the first statement, and also gives a homeomorphism \[\msf{Zg}(\T)\setminus\{X\in\msf{Zg}(\T):FX=0\}\simeq\msf{Zg}(\Flat{\U^\c}) \cap \msf{Def}(\msf{Im}(\widehat{F}\y)).\] By combining this with the homeomorphism $\y\colon\msf{Zg}(\U) \cap \msf{Def}(\msf{Im}\,F)\xrightarrow{\sim}\msf{Zg}(\Flat{\U^\c}) \cap \msf{Def}(\msf{Im}(\hat{F}\y))$ from \cref{purityflats}\ref{purityflats4}, the second statement follows.
\end{proof}

\begin{rem}
As in \cref{allresultstransfertocoherent}, we can also transfer the other results concerning functoriality of the Ziegler spectrum, as stated in \cite[\S15]{dac}, to definable functors between triangulated categories. We note that the above also recovers \cite[Theorem 6.13]{Beligiannis}, in fact, any fully faithful embedding yields an embedding of Ziegler spectra. 
\end{rem}

Let us now give some examples of the previous result.
\begin{chunk}
Let $f\colon R\to S$ be a map of rings. Recall from \cite{geiglelenzing} that $f$ is a homological ring epimorphism if the restriction of scalars functor $f^{*}\colon\D(S)\to\D(R)$ is fully faithful. Since restriction of scalars is a definable functor, we therefore obtain a closed and continuous embedding $f^{*}\colon\msf{Zg}(\D(S))\hookrightarrow\msf{Zg}(\D(R))$ which is a homeomorphism onto its image.
\end{chunk}

\begin{chunk}
If $L$ is a smashing localization of $\T$, then the inclusion of the local objects $i\colon L\T \hookrightarrow \T$ is a fully faithful definable functor. In particular, it induces a closed, continuous embedding of Ziegler spectra $\msf{Zg}(L\T) \hookrightarrow \msf{Zg}(\T)$ which is a homeomorphism onto its image.
\end{chunk}

\bibliographystyle{abbrv}
\bibliography{references.bib}

\begin{thebibliography}{10}

\bibitem{Adams}
J.~F. Adams.
\newblock A variant of {E}. {H}. {B}rown's representability theorem.
\newblock {\em Topology}, 10:185--198, 1971.

\bibitem{AMV}
L.~Angeleri~H\"{u}gel, F.~Marks, and J.~Vit\'{o}ria.
\newblock Torsion pairs in silting theory.
\newblock {\em Pacific J. Math.}, 291(2):257--278, 2017.

\bibitem{discrete}
K.~K. Arnesen, R.~Laking, D.~Pauksztello, and M.~Prest.
\newblock The {Z}iegler spectrum for derived-discrete algebras.
\newblock {\em Adv. Math.}, 319:653--698, 2017.

\bibitem{Balmerhomsupp}
P.~Balmer.
\newblock Homological support of big objects in tensor-triangulated categories.
\newblock {\em J. \'{E}c. polytech. Math.}, 7:1069--1088, 2020.

\bibitem{Balmernilpotence}
P.~Balmer.
\newblock Nilpotence theorems via homological residue fields.
\newblock {\em Tunis. J. Math.}, 2(2):359--378, 2020.

\bibitem{BalmerCameron}
P.~Balmer and J.~C. Cameron.
\newblock Computing homological residue fields in algebra and topology.
\newblock {\em Proc. Amer. Math. Soc.}, 149(8):3177--3185, 2021.

\bibitem{BDS}
P.~Balmer, I.~Dell'Ambrogio, and B.~Sanders.
\newblock Grothendieck-{N}eeman duality and the {W}irthm\"{u}ller isomorphism.
\newblock {\em Compos. Math.}, 152(8):1740--1776, 2016.

\bibitem{bks}
P.~Balmer, H.~Krause, and G.~Stevenson.
\newblock Tensor-triangular fields: ruminations.
\newblock {\em Selecta Math. (N.S.)}, 25(1):Paper No. 13, 36, 2019.

\bibitem{bel}
A.~Beligiannis.
\newblock Relative homological algebra and purity in triangulated categories.
\newblock {\em J. Algebra}, 227(1):268--361, 2000.

\bibitem{Beligiannis}
A.~Beligiannis.
\newblock Auslander-{R}eiten triangles, {Z}iegler spectra and {G}orenstein
  rings.
\newblock {\em $K$-Theory}, 32(1):1--82, 2004.

\bibitem{BIKP}
D.~Benson, S.~B. Iyengar, H.~Krause, and J.~Pevtsova.
\newblock Local duality for representations of finite group schemes.
\newblock {\em Compos. Math.}, 155(2):424--453, 2019.

\bibitem{bg}
D.~J. Benson and G.~P. Gnacadja.
\newblock Phantom maps and purity in modular representation theory. {I}.
\newblock volume 161, pages 37--91. 1999.
\newblock Algebraic topology (Kazimierz Dolny, 1997).

\bibitem{birdgfc}
I.~Bird.
\newblock Purity and ascent for {G}orenstein flat cotorsion modules.
\newblock Updated version in preparation, 2025.

\bibitem{dualitypairs}
I.~Bird and J.~Williamson.
\newblock Duality pairs, phantom maps, and definability in triangulated
  categories.
\newblock {\em Proc. Roy. Soc. Edinburgh Sect. A}.
\newblock Published online 2024:1-46.
  \url{https://doi.org/10.1017/prm.2024.73}.

\bibitem{birdwilliamsonhomological}
I.~Bird and J.~Williamson.
\newblock The homological spectrum via definable subcategories.
\newblock \emph{Bull. Lond. Math. Soc.} (to appear),
  \url{https://doi.org/10.1112/blms.70014}, 2025.

\bibitem{BWZ}
I.~Bird, J.~Williamson, and A.~Zvonareva.
\newblock The shift-homological spectrum and parametrising kernels of rank
  functions.
\newblock arXiv:2502.11939, 2025.

\bibitem{CKN}
J.~D. Christensen, B.~Keller, and A.~Neeman.
\newblock Failure of {B}rown representability in derived categories.
\newblock {\em Topology}, 40(6):1339--1361, 2001.

\bibitem{dcmca}
L.~W. Christensen, H.-B. Foxby, and H.~Holm.
\newblock Derived category methods in commutative algebra.
\newblock Draft version of 11/21. Obtained at
  \url{https://www.math.ttu.edu/~lchriste/download/dcmca.pdf}.

\bibitem{chuanglazarev}
J.~Chuang and A.~Lazarev.
\newblock Rank functions on triangulated categories.
\newblock {\em J. Reine Angew. Math.}, 781:127--164, 2021.

\bibitem{conde2022functorial}
T.~Conde, M.~Gorsky, F.~Marks, and A.~Zvonareva.
\newblock A functorial approach to rank functions on triangulated categories.
\newblock {\em J. Reine Angew. Math.}, 811:135--181, 2024.

\bibitem{CrawleyBoevey}
W.~Crawley-Boevey.
\newblock Locally finitely presented additive categories.
\newblock {\em Comm. Algebra}, 22(5):1641--1674, 1994.

\bibitem{prestgarkusha}
G.~Garkusha and M.~Prest.
\newblock Triangulated categories and the {Z}iegler spectrum.
\newblock {\em Algebr. Represent. Theory}, 8(4):499--523, 2005.

\bibitem{geiglelenzing}
W.~Geigle and H.~Lenzing.
\newblock Perpendicular categories with applications to representations and
  sheaves.
\newblock {\em J. Algebra}, 144(2):273--343, 1991.

\bibitem{gillespie}
J.~Gillespie.
\newblock The flat stable module category of a coherent ring.
\newblock {\em J. Pure Appl. Algebra}, 221(8):2025--2031, 2017.

\bibitem{gregprest}
L.~Gregory and M.~Prest.
\newblock Representation embeddings, interpretation functors and controlled
  wild algebras.
\newblock {\em J. Lond. Math. Soc. (2)}, 94(3):747--766, 2016.

\bibitem{kredc}
H.~Krause.
\newblock Exactly definable categories.
\newblock {\em J. Algebra}, 201(2):456--492, 1998.

\bibitem{Krauselfp}
H.~Krause.
\newblock Functors on locally finitely presented additive categories.
\newblock {\em Colloq. Math.}, 75(1):105--132, 1998.

\bibitem{krsmash}
H.~Krause.
\newblock Smashing subcategories and the telescope conjecture---an algebraic
  approach.
\newblock {\em Invent. Math.}, 139(1):99--133, 2000.

\bibitem{krcoh}
H.~Krause.
\newblock Coherent functors in stable homotopy theory.
\newblock {\em Fund. Math.}, 173(1):33--56, 2002.

\bibitem{krcq}
H.~Krause.
\newblock Cohomological quotients and smashing localizations.
\newblock {\em Amer. J. Math.}, 127(6):1191--1246, 2005.

\bibitem{Krausenotes}
H.~Krause.
\newblock Localization theory for triangulated categories.
\newblock In {\em Triangulated categories}, volume 375 of {\em London Math.
  Soc. Lecture Note Ser.}, pages 161--235. Cambridge Univ. Press, Cambridge,
  2010.

\bibitem{krbook}
H.~Krause.
\newblock {\em Homological theory of representations}, volume 195 of {\em
  Cambridge Studies in Advanced Mathematics}.
\newblock Cambridge University Press, Cambridge, 2022.

\bibitem{laking}
R.~Laking.
\newblock Purity in compactly generated derivators and t-structures with
  {G}rothendieck hearts.
\newblock {\em Math. Z.}, 295(3-4):1615--1641, 2020.

\bibitem{lv}
R.~Laking and J.~Vit\'{o}ria.
\newblock Definability and approximations in triangulated categories.
\newblock {\em Pacific J. Math.}, 306(2):557--586, 2020.

\bibitem{LN}
J.~Lipman and A.~Neeman.
\newblock Quasi-perfect scheme-maps and boundedness of the twisted inverse
  image functor.
\newblock {\em Illinois J. Math.}, 51(1):209--236, 2007.

\bibitem{HA}
J.~{Lurie}.
\newblock {\em {Higher algebra}}.
\newblock Avaliable from the author's webpage at
  \url{http://www.math.harvard.edu/~lurie/papers/HA.pdf}, September 2017.

\bibitem{Neemanduality}
A.~Neeman.
\newblock The {G}rothendieck duality theorem via {B}ousfield's techniques and
  {B}rown representability.
\newblock {\em J. Amer. Math. Soc.}, 9(1):205--236, 1996.

\bibitem{NeemanBrownAdams}
A.~Neeman.
\newblock On a theorem of {B}rown and {A}dams.
\newblock {\em Topology}, 36(3):619--645, 1997.

\bibitem{NSZ}
P.~Nicol\'{a}s, M.~Saor\'{\i}n, and A.~Zvonareva.
\newblock Silting theory in triangulated categories with coproducts.
\newblock {\em J. Pure Appl. Algebra}, 223(6):2273--2319, 2019.

\bibitem{prestinterp}
M.~Prest.
\newblock Interpreting modules in modules.
\newblock {\em Ann. Pure Appl. Logic}, 88(2-3):193--215, 1997.

\bibitem{psl}
M.~Prest.
\newblock {\em Purity, spectra and localisation}, volume 121 of {\em
  Encyclopedia of Mathematics and its Applications}.
\newblock Cambridge University Press, Cambridge, 2009.

\bibitem{dac}
M.~Prest.
\newblock Definable additive categories: purity and model theory.
\newblock {\em Mem. Amer. Math. Soc.}, 210(987):vi+109, 2011.

\bibitem{saorstov}
M.~Saor\'in and J.~Šťovíček.
\newblock {$t$}-structures with {G}rothendieck hearts via functor categories.
\newblock {\em Selecta Math. (N.S.)}, 29(5):Paper No. 77, 73, 2023.

\bibitem{spure}
J.~{\v{S}}\v{t}ov\'{\i}\v{c}ek.
\newblock On purity and applications to coderived and singularity categories.
\newblock arXiv.1412.1615, 2014.

\bibitem{toen}
B.~To\"{e}n and M.~Vaqui\'{e}.
\newblock Moduli of objects in dg-categories.
\newblock {\em Ann. Sci. \'{E}cole Norm. Sup. (4)}, 40(3):387--444, 2007.

\end{thebibliography}
\end{document}